\documentclass[11pt]{article}

\usepackage[utf8]{inputenc}
\usepackage{amsmath, amsthm, amssymb}
\usepackage{enumerate}
\usepackage{pdflscape}
\usepackage{caption}
\usepackage{bm}

\usepackage{ifpdf}
\ifpdf
\usepackage[pdftex]{graphicx}
\else
\usepackage[dvips]{graphicx}
\fi
\usepackage{tikz}
 	 \usetikzlibrary{arrows,backgrounds}
\usepackage[all]{xy}

\usepackage{multicol}

\usepackage{tocvsec2}

\usepackage{bbm}

\input xy
\xyoption{all}

\usepackage[pdftex,plainpages=false,hypertexnames=false,pdfpagelabels]{hyperref}
\newcommand{\arxiv}[1]{\href{http://arxiv.org/abs/#1}{\tt arXiv:\nolinkurl{#1}}}
\newcommand{\arXiv}[1]{\href{http://arxiv.org/abs/#1}{\tt arXiv:\nolinkurl{#1}}}
\newcommand{\doi}[1]{\href{http://dx.doi.org/#1}{{\tt DOI:#1}}}

\newcommand{\googlebooks}[1]{(preview at \href{http://books.google.com/books?id=#1}{google books})}

\usepackage{xcolor}
\definecolor{dark-red}{rgb}{0.7,0.25,0.25}
\definecolor{dark-blue}{rgb}{0.15,0.15,0.55}
\definecolor{medium-blue}{rgb}{0,0,.8}
\definecolor{DarkGreen}{RGB}{0,150,0}
\definecolor{rho}{named}{red}
\hypersetup{
   colorlinks, linkcolor={purple},
   citecolor={medium-blue}, urlcolor={medium-blue}
}

\usepackage{longtable}
\usepackage{fullpage}

\theoremstyle{plain}
\newtheorem{thm}{Theorem}[section]
\newtheorem*{thm*}{Theorem}
\newtheorem{thmalpha}{Theorem}

\newtheorem*{cor*}{Corollary}

\newtheorem*{conj*}{Conjecture}
\newtheorem{lem}[thm]{Lemma}
\newtheorem{prop}[thm]{Proposition}

\newtheorem*{quest*}{Question}
\newtheorem*{claim*}{Claim}

\theoremstyle{definition}
\newtheorem{defn}[thm]{Definition}

\newtheorem{alg}[thm]{Algorithm}

\newtheorem{nota}[thm]{Notation}

\newtheorem{ex}[thm]{Example}
\newtheorem{prob}[thm]{Problem}
\newtheorem{fact}[thm]{Fact}

\newtheorem{sub-ex}[thm]{Sub-Example}
\newtheorem{rem}[thm]{Remark}
\newtheorem*{rem*}{Remark}

\DeclareMathOperator{\Aut}{Aut}
\DeclareMathOperator{\coev}{coev}
\DeclareMathOperator{\diag}{diag}

\DeclareMathOperator{\End}{End}

\DeclareMathOperator{\ev}{ev}
\DeclareMathOperator{\loc}{loc}
\DeclareMathOperator{\Hom}{Hom}

\DeclareMathOperator{\Out}{Out}

\DeclareMathOperator{\id}{id}
\DeclareMathOperator{\Inv}{Inv}

\DeclareMathOperator{\Ind}{Ind}

\DeclareMathOperator{\Irr}{Irr}

\DeclareMathOperator{\Spin}{Spin}
\DeclareMathOperator{\SU}{SU}
\DeclareMathOperator{\Stab}{Stab}

\DeclareMathOperator{\triv}{triv}

\DeclareMathOperator{\res}{res}
\DeclareMathOperator{\Res}{Res}

\newcommand{\comment}[1]{}

\newcommand{\be}{\begin{enumerate}[label=(\arabic*)]}
\newcommand{\ee}{\end{enumerate}}

\newcommand{\set}[2]{\left\{#1 ~\middle|~ #2\right\}}

\renewcommand{\loc}{{\rm{loc}}}

\def\semicolon{;}
\def\applytolist#1{
    \expandafter\def\csname multi#1\endcsname##1{
        \def\multiack{##1}\ifx\multiack\semicolon
            \def\next{\relax}
        \else
            \csname #1\endcsname{##1}
            \def\next{\csname multi#1\endcsname}
        \fi
        \next}
    \csname multi#1\endcsname}

\def\calc#1{\expandafter\def\csname c#1\endcsname{{\mathcal #1}}}
\applytolist{calc}QWERTYUIOPLKJHGFDSAZXCVBNM;
\def\bbc#1{\expandafter\def\csname bb#1\endcsname{{\mathbb #1}}}
\applytolist{bbc}QWERTYUIOPLKJHGFDSAZXCVBNM;
\def\bfc#1{\expandafter\def\csname bf#1\endcsname{{\mathbf #1}}}
\applytolist{bfc}QWERTYUIOPLKJHGFDSAZXCVBNM;
\def\sfc#1{\expandafter\def\csname s#1\endcsname{{\sf #1}}}
\applytolist{sfc}QWERTYUIOPLKJHGFDSAZXCVBNM;
\def\fc#1{\expandafter\def\csname f#1\endcsname{{\mathfrak #1}}}
\applytolist{fc}QWERTYUIOPLKJHGFDSAZXCVBNM;

\newcommand{\Rep}{{\sf Rep}}

\renewcommand{\Vec}{{\sf Vec}}
\newcommand{\fdVec}{{\sf Vec_{fd}}}

\newcommand{\EqBr}{{\sf Aut}_{\otimes}^{\sf br}}
\newcommand{\uEqBr}{\underline{\sf Aut}_{\otimes}^{\sf br}}
\newcommand{\uEq}{\underline{\sf Aut}}
\newcommand{\Eq}{{\sf Aut}}
\newcommand{\uEqTens}{\underline{\sf Aut}_\otimes}
\newcommand{\EqTens}{{\sf Aut}_\otimes}
\newcommand{\sfG}{{\sf G}}
\newcommand{\usfG}{\underline{\sf G}}

\newcommand{\uG}{\underline{G}}

\renewcommand{\Stab}{{\sf Stab}}

\newcommand{\EqStab}{\Stab_\otimes}
\newcommand{\uEqStab}{\underline{\Stab}_\otimes}
\newcommand{\EqBrStab}{{\Stab_\otimes^{\sf br}}}
\newcommand{\uEqBrStab}{\underline{\Stab}_\otimes^{\sf br}}

\newcommand{\noshow}[1]{}
\newcommand{\MR}[1]{}

\usepackage{tikz-cd}
\usetikzlibrary{shapes}
\usetikzlibrary{cd}
\usetikzlibrary{backgrounds}
\usetikzlibrary{decorations,decorations.pathreplacing,decorations.markings,decorations.pathmorphing}
\usetikzlibrary{fit,calc,through}
\usetikzlibrary{external}
\usetikzlibrary{arrows}
\tikzset{vertex/.style = {shape=circle,draw,fill=black,inner sep=0pt,minimum size=5pt}}
\tikzset{edge/.style = {->,> = latex', bend right}}
\tikzset{
	super thick/.style={line width=3pt}
}
\tikzset{
    quadruple/.style args={[#1] in [#2] in [#3] in [#4]}{
        #1,preaction={preaction={preaction={draw,#4},draw,#3}, draw,#2}
    }
}
\tikzstyle{shaded}=[fill=red!10!blue!20!gray!30!white]
\tikzstyle{unshaded}=[fill=white]
\tikzstyle{empty box}=[circle, draw, thick, fill=white, opaque, inner sep=2mm]
\tikzstyle{annular}=[scale=.7, inner sep=1mm, baseline]
\tikzstyle{rectangular}=[scale=.75, inner sep=1mm, baseline=-.1cm]
\tikzstyle{mid>}=[decoration={markings, mark=at position 0.5 with {\arrow{>}}}, postaction={decorate}]
\tikzstyle{mid<}=[decoration={markings, mark=at position 0.5 with {\arrow{<}}}, postaction={decorate}]
\tikzstyle{over}=[double, draw=white, super thick, double=]

\newcommand{\roundNbox}[6]{
	\draw[rounded corners=5pt, very thick, #1] ($#2+(-#3,-#3)+(-#4,0)$) rectangle ($#2+(#3,#3)+(#5,0)$);
	\coordinate (ZZa) at ($#2+(-#4,0)$);
	\coordinate (ZZb) at ($#2+(#5,0)$);
	\node at ($1/2*(ZZa)+1/2*(ZZb)$) {#6};
}

\begin{document}
\title{Spontaneous symmetry breaking from anyon condensation}
\author{Marcel Bischoff, 
Corey Jones, 
Yuan-Ming Lu, 
and 
David Penneys
}
\date{}
\maketitle

\begin{abstract}
In a physical system undergoing a continuous quantum phase transition, spontaneous symmetry breaking occurs when certain symmetries of the Hamiltonian fail to be preserved in the ground state.
In the traditional Landau theory, a symmetry group can break down to any subgroup. 
However, this no longer holds across a continuous phase transition driven by anyon condensation in symmetry enriched topological orders (SETOs). 
For a SETO described by a $G$-crossed braided extension $\mathcal{C}\subseteq \mathcal{C}^{\times}_{G}$, we show that physical considerations require that a connected \'etale algebra $A\in \cC$ admit a $G$-equivariant algebra structure for symmetry to be preserved under condensation of $A$. 
Given any categorical action $\underline{G}\rightarrow \underline{\sf Aut}_{\otimes}^{\sf br}(\cC)$ such that $g(A)\cong A$ for all $g\in G$, we show there is a short exact sequence whose splittings correspond to $G$-equivariant algebra structures.
The non-splitting of this sequence forces spontaneous symmetry breaking under condensation of $A$. 
\textit{}Furthermore, we show that if symmetry is preserved, there is a canonically associated SETO of $\mathcal{C}^{\operatorname{loc}}_{A}$, and gauging this symmetry commutes with anyon condensation.
\end{abstract}

\setcounter{tocdepth}{2}
\tableofcontents

\section{Introduction}

A cornerstone of condensed matter physics is the concept of spontaneous symmetry breaking exemplified by Landau's theory of symmetry breaking phase transitions \cite{Landau1937b,Landau1937c}. 
It was long believed that all phases of matter are classified by their symmetries until the discovery of topological orders \cite{Wen1989a,Wen2004B} in fractional quantum Hall fluids \cite{Tsui1982}. 
The presence of anyons \cite{Wilczek1990B}, i.e., particles obeying neither Bose nor Fermi statistics in two spatial dimensions, gives rise to an extremely rich structure beyond the Landau paradigm. 
In particular, the interplay between symmetries and topological orders leads to a large class of new phases of matter called \emph{symmetry enriched topological orders} (SETOs) \cite{Wen2002,Chen2013,Essin2013,Mesaros2013,Hung2013,Hung2014,1410.4540,MR3367540,Lu2016,Tarantino2016,Chiu2016,Chen2017,MR3613518,Lan2017a}.

According to Landau theory, in a physical system whose Hamiltonian preserves a symmetry group $G$, the symmetry of its ground state can spontaneously break down to any subgroup $H\subset G$. Within Landau theory, continuous quantum phase transitions happen between one ground state whose unbroken symmetry is a subgroup of the other, and the associated critical phenomena is described by the Ginzburg-Landau theory. It is natural to wonder how this picture is altered by the interplay of symmetry and topology in SETOs. The development of a theory for continuous phase transitions involving anyons and topological orders beyond Landau's paradigm raises a fundamental question for modern condensed matter physics. 

In this work, we make a first step to resolve this issue, focusing on quantum phase transitions between different gapped SETOs driven by anyon condensation \cite{Bais2009,MR3246855,Neupert2016,Neupert2016a,Burnell2018}. 
We develop a categorical framework for spontaneous symmetry breaking driven by condensing anyons in topological orders. 
This perspective reveals a set of necessary conditions for a symmetry to be preserved in continuous quantum phase transitions between two different SETOs. 
Given an arbitrary SETO, our categorical framework allows us to classify all possible gapped phases connected to this SETO by a continous quantum phase transition driven by the condensation of anyons, revealing the structure of a phase diagram involving all SETOs. 

We now recall the mathematical descriptions of anyon condensation and SETOs. The topological order of a gapped system is described by a unitary modular tensor category (UMTC) $\cC$. 
The simple objects of the category describe the superselection sectors (anyons) of quasi-particle excitations, while the data of the category describes the fusion and braiding operators of the anyons. 
If the system has a global on-site $G$-symmetry, then there is a natural SETO associated to the system, described by a \emph{G-crossed braided extension} \cite{1410.4540} (see also \cite[\S8.24]{MR3242743}). 
In particular, we have a categorical action $\uG\rightarrow \uEqBr(\cC)$, where the later denotes the categorical group of braided autoequivalences of $\cC$.
Condensible anyons are described by connected \'etale algebra objects $A\in \cC$ \cite{MR3246855}. 
On the mathematical level, anyon condensation is described as passage from the category $\cC$ to the category $\cC^{\operatorname{loc}}_{A}$ of local $A$ modules, which is again a UMTC. The category $\cC^{\operatorname{loc}}_{A}$ describes the topological order of the phase post transition.

A symmetry is spontaneously broken across a phase transition if it fails to be preserved in the ground state post transition. 
In particular, if a symmetry is unbroken (\emph{preserved}), then it should also act consistently on the system after the phase transition. In general, non-universal properties of the microscopic Hamiltonian (i.e.\ energetics) can cause a symmetry to be broken that are beyond the description of category theory, e.g., the symmetry breaking that occurs in Landau theory.
Thus from the categorical point of view, one can only determine \emph{sufficient} conditions for symmetry breaking which are universal and independent of the specific Hamiltonian. 
We refer to symmetry broken for such categorical reasons as \textit{categorically broken symmetry}.
From this point of view, it is natural to consider categorically \emph{necessary} conditions for symmetry to be preserved under a phase transition driven by anyon condensation.

We analyze this problem mathematically in the more general setup of a non-degenerate braided fusion category $\cC$, where we require neither unitarity nor a spherical structure, together with a connected \'{e}tale (commutative and separable) algebra $A\in \cC$.
We prove there are two \emph{obstructions} for a categorical symmetry $\uG \to \uEqBr(\cC)$ to be preserved under anyon condensation:
\begin{enumerate}[(1)]
\item 
For all $g\in G$, $g(A)\cong A$ \emph{as algebras}, and
\item
There is a $G$-equivariant algebra structure $\lambda = \{\lambda^g: g(A) \to A\}_{g\in G}$ on $A$.
\end{enumerate}
Violating either condition will necessarily break the symmetry group $G$ after the condensation of $A$. 
In particular, as we will show later, the second condition can be concisely expressed as the existence of a splitting of a short exact sequence \eqref{eq:SESintro} below, where inequivalent splittings correspond to different $G$-symmetric phases obtained by condensing $A$.

We now describe the mathematical and physical reasoning that leads to these obstructions.


Suppose 
$\cF$ is a $G$-crossed braided extension of $\cC$, with associated categorical action $\uG \to \uEqBr(\cC)$.
When we condense our \'{e}tale algebra $A\in \cC$, this algebra describes the vacuum sector of the new phase.
A categorically necessary condition to preserve symmetry is that the symmetry must be preserved in the vacuum.
For $g\in \uEqBr(\cC)$, to act by a monoidal transformation on the new phase, we must have an algebra isomorphism $\lambda^g:g(A)\to A$, and this algebra isomorphism gives us an isomorphism for each $g$-graded defect $X\in \cF_g$:
$$
\begin{tikzpicture}[baseline=-.1cm]
  \draw (-.3,.6) -- (-.3,1);
  \draw (.3,-1) -- (.3,-.7) .. controls ++(90:.4cm) and ++(270:.4cm) .. (-.3,0);
  \draw[super thick, white] (-.3,-1) -- (-.3,-.7) .. controls ++(90:.4cm) and ++(270:.4cm) .. (.3,0) -- (.3,1);
  \draw (-.3,-1) -- (-.3,-.7) .. controls ++(90:.4cm) and ++(270:.4cm) .. (.3,0) -- (.3,1);
  \roundNbox{}{(-.3,.3)}{.3}{0}{0}{$\lambda^g$}
  \node at (-.5,-.8) {\scriptsize{$X$}};
  \node at (.5,.8) {\scriptsize{$X$}};
  \node at (.5,-.8) {\scriptsize{$A$}};
  \node at (-.7,-.2) {\scriptsize{$g(A)$}};
  \node at (-.5,.8) {\scriptsize{$A$}};
\end{tikzpicture}
=
(\lambda^g \otimes \id_X)
\circ
\beta_{X,A}
:
X\otimes A \to A\otimes X.
$$
The multiplication $m$ of $A$ should be compatible with these isomorphisms
$$
\begin{tikzpicture}[baseline=-.1cm]
  \draw (-.3,.6) -- (-.3,1);
  \draw (.3,-.7) .. controls ++(90:.4cm) and ++(270:.4cm) .. (-.3,0);
  \draw (0,-1) arc (180:0:.3cm);
  \filldraw (.3,-.7) circle (.05cm);
  \draw[super thick, white] (-.3,-1) -- (-.3,-.7) .. controls ++(90:.4cm) and ++(270:.4cm) .. (.3,0) -- (.3,1);
  \draw (-.3,-1) -- (-.3,-.7) .. controls ++(90:.4cm) and ++(270:.4cm) .. (.3,0) -- (.3,1);
  \roundNbox{}{(-.3,.3)}{.3}{0}{0}{$\lambda^g$}
  \node at (-.3,-1.2) {\scriptsize{$X$}};
  \node at (.3,1.2) {\scriptsize{$X$}};
  \node at (0,-1.2) {\scriptsize{$A$}};
  \node at (.6,-1.2) {\scriptsize{$A$}};
  \node at (.3,-.85) {\scriptsize{$m$}};
  \node at (-.7,-.2) {\scriptsize{$g(A)$}};
  \node at (-.3,1.2) {\scriptsize{$A$}};
\end{tikzpicture}
=
\begin{tikzpicture}[baseline=-.1cm]
  \draw (-.3,.8) -- (-.3,1);
  \draw (.6,-1) .. controls ++(90:.4cm) and ++(270:.4cm) .. (0,0);
  \draw (0,-1) .. controls ++(90:.4cm) and ++(270:.4cm) .. (-.6,0);
  \draw (-.6,.5) arc (180:0:.3cm);
  \filldraw (-.3,.8) circle (.05cm);
  \draw[super thick, white] (-.6,-1) .. controls ++(90:.6cm) and ++(270:.6cm) .. (.6,0) -- (.6,1);
  \draw (-.6,-1) .. controls ++(90:.6cm) and ++(270:.6cm) .. (.6,0) -- (.6,1);
  \roundNbox{}{(0,.25)}{.25}{0}{0}{$\lambda^g$}
  \roundNbox{}{(-.6,.25)}{.25}{0}{0}{$\lambda^g$}
  \node at (-.6,-1.2) {\scriptsize{$X$}};
  \node at (.6,1.2) {\scriptsize{$X$}};
  \node at (0,-1.2) {\scriptsize{$A$}};
  \node at (.6,-1.2) {\scriptsize{$A$}};
  \node at (-.3,.65) {\scriptsize{$m$}};
  \node at (-.9,-.2) {\scriptsize{$g(A)$}};
  \node at (-.3,1.2) {\scriptsize{$A$}};
\end{tikzpicture}
\qquad\qquad
\forall X\in \cF_G,
$$
and our isomorphisms $\lambda$ should be \emph{coherent} with the $G$-crossed braided extension $\cF$.
Thus for a $g$-graded defect $X\in \cF_g$ and an $h$-graded defect $y\in \cF_h$, we have
$$
\begin{tikzpicture}[baseline=.7cm]
  \draw (-1.2,2.2) -- (-1.2,2.6);
  \draw (-.3,.6) -- (-.3,.7) .. controls ++(90:.4cm) and ++(270:.4cm) .. (-1.2,1.6);
  \draw (.3,-1) -- (.3,-.7) .. controls ++(90:.4cm) and ++(270:.4cm) .. (-.3,0);
  \draw[super thick, white] (-.3,-1) -- (-.3,-.7) .. controls ++(90:.4cm) and ++(270:.4cm) .. (.3,0) -- (.3,2);
  \draw (-.3,-1) -- (-.3,-.7) .. controls ++(90:.4cm) and ++(270:.4cm) .. (.3,0) -- (.3,2.6);
  \draw[super thick, white] (-1.2,-1) -- (-1.2,.7) .. controls ++(90:.4cm) and ++(270:.4cm) .. (-.3,1.6) -- (-.3,2.6);
  \draw (-1.2,-1) -- (-1.2,.7) .. controls ++(90:.4cm) and ++(270:.4cm) .. (-.3,1.6) -- (-.3,2.6);
  \roundNbox{}{(-.3,.3)}{.3}{0}{0}{$\lambda^h$}
  \roundNbox{}{(-1.2,1.9)}{.3}{0}{0}{$\lambda^g$}
  \node at (-1.4,-.8) {\scriptsize{$X$}};
  \node at (-.5,-.8) {\scriptsize{$Y$}};
  \node at (.5,-.8) {\scriptsize{$A$}};
  \node at (-.7,-.2) {\scriptsize{$h(A)$}};
  \node at (-.1,.8) {\scriptsize{$A$}};
  \node at (-1.6,1.4) {\scriptsize{$g(A)$}};
  \node at (-1.4,2.4) {\scriptsize{$A$}};
  \node at (-.5,2.4) {\scriptsize{$X$}};
  \node at (.5,2.4) {\scriptsize{$Y$}};
\end{tikzpicture}
=
\begin{tikzpicture}[baseline=-.1cm]
  \draw (-.3,.6) -- (-.3,1);
  \draw (.3,-1) -- (.3,-.7) .. controls ++(90:.4cm) and ++(270:.4cm) .. (-.3,0);
  \draw[super thick, white] (-.3,-1) -- (-.3,-.7) .. controls ++(90:.4cm) and ++(270:.4cm) .. (.3,0) -- (.3,1);
  \draw (-.3,-1) -- (-.3,-.7) .. controls ++(90:.4cm) and ++(270:.4cm) .. (.3,0) -- (.3,1);
  \roundNbox{}{(-.3,.3)}{.3}{.1}{.1}{$\lambda^{gh}$}
  \node at (-.8,-.8) {\scriptsize{$X\otimes Y$}};
  \node at (.8,.8) {\scriptsize{$X\otimes Y$}};
  \node at (.5,-.8) {\scriptsize{$A$}};
  \node at (-.8,-.2) {\scriptsize{$gh(A)$}};
  \node at (-.5,.8) {\scriptsize{$A$}};
\end{tikzpicture}
\qquad\qquad
\forall X\in \cF_g, Y\in \cF_h.
$$
Under the coherence conditions for $G$-crossed braided fuison categories, this is exactly the condition that the following diagram commutes for all $g,h\in G$:
$$
\begin{tikzcd}
g(h(A)) \arrow{r}{g(\lambda^{h})} \arrow{d}[swap]{\mu^{A}_{g,h}} 
& 
g(A) \arrow{d}{\lambda^{g}} 
\\
gh(A) \arrow{r}{\lambda^{gh}} 
& 
A.
\end{tikzcd}
$$
That is, $\{\lambda^g : g(A) \to A\}_{g\in G}$ endows $A$ with the structure of a $G$-\emph{equivariant algebra object} in the \emph{equivariantization} $\cC^G$ (see Definition \ref{defn:Equivariantization} below).

\begin{defn}
Let $\underline{G}\rightarrow \uEqBr(\cC)$ be a categorical action and $A\in \cC$ a connected \'etale algebra. 
We say that \textit{symmetry is categorically preserved} if there exists a $G$-equivariant algebra structure on $A$. 
We say that \textit{symmetry is categorically broken} if no such structure exists.
\end{defn}

If symmetry is categorically broken, then it forces the symmetry to break across the anyon-condensation transition in any microscopic realization of this SETO. 
We now describe some mathematical consequences of symmetry being categorically preserved, which justify our definition.

Define the categorical groups
\begin{align*}
\uEqBr(\cC|A)&:= \set{(\alpha, \lambda)}{\alpha \in \uEqBr(\cC)\text{ and }\lambda: \alpha(A) \to A\text{ is an algebra isomorphism}}
\\
\uEqBrStab(A) &:=\set{\alpha\in \uEqBr(\cC)}{\alpha(A) \cong A \text{ as algebras}}
\end{align*}
While the objects of these two categorical groups look similar, and involve braided autoequivalences $\alpha$ which preserve $A$ as an algebra, there is a subtle difference.
To describe an object in $\uEqBr(\cC|A)$ we must choose a \emph{specific} algebra isomorphism $\lambda: \alpha(A) \to A$, while to describe an object in $\uEqBrStab(A)$, we merely require the existence of such a $\lambda$.
The morphisms between such pairs are monoidal natural transformations which  
must be compatible with these specific choices of algebra isomorphisms.
We refer the reader to Definition \ref{defn:ExamplesOfCategoricalGroups} below for the preceise definitions of these morphisms.
However, observe that there is an obvious forgetful monoidal functor $\underline{F}:\uEqBr(\cC|A) \to \uEqBrStab(A) \subset \uEqBr(\cC)$ which forgets the specific choice of $\lambda$.

\begin{thmalpha}
\label{thmalpha:LiftsAndGEquivariantStructure}
There is a canonical strong monoidal functor $\underline{F}_A:\uEqBr(\cC|A) \to \uEqBr(\cC_A^\loc)$.
Given any categorical action $(\underline{\rho},\mu):\uG \to \uEqBr(\cC)$ whose image lies in $\uEqBrStab(A)$, lifts to $\uEqBr(\cC|A)$ are in bijective correspondence with $G$-equivariant algebra structures
$\lambda = \{\lambda^g : g(A) \to A\}_{g\in G}$.
Thus a $G$-equivariant algebra structure induces a canonical categorical action $\uG \to \uEqBr(\cC_A^\loc)$.
\end{thmalpha}

\begin{equation*}
\begin{tikzcd}
\uG
\ar[]{dr}{{(\underline{\rho},\mu)}}
\ar[dd,dashed, "\exists?"]
\\
& 
\uEqBrStab(A)
\ar[hook]{r}{}
&
\uEqBr(\cC)
\\
\uEqBr(\cC|A)
\ar[swap, two heads]{ur}{\underline{F}}
\ar[dr, "\underline{F}_A"]
\\
&
\uEqBr(\cC_A^\loc)
\end{tikzcd}
\end{equation*}
We refer the reader to \S\ref{sec:InducedActionForEtaleAlgebra} for more details and the proof of Theorem \ref{thmalpha:LiftsAndGEquivariantStructure}.
This theorem is telling us that if symmetry is categorically preserved, then the group $G$ acts consistently by braided autoequivalences on the phase $\cC_A^{\loc}$ after condensation.
Physically, this means there is no categorical reason which forces spontaneous symmetry breaking.

Now having an action of $G$ on $\cC_A^{\loc}$ is not sufficient to specify a SETO. 
In general, we need an entire $G$-crossed braided extension of $\cC$. 
Given a categorical action $(\underline{\rho},\mu):\underline{G}\rightarrow\uEqBr(\cC_A^{\loc})$, there is an obstruction to the existence of such an extension $o_{4}\in H^{4}(G,U(1))$ \cite{MR2677836}, which is known as the \emph{anomaly} of the symmetry fractionalization class $(\underline{\rho},\mu)$\cite{1410.4540}.
In Theorem \ref{thmalpha:GCrossedBraidedFromGEquivariant} below, we show that if this anomaly vanishes for the action on $\cC$, then it must also vanish for the action on $\cC_{A}^{loc}$ constructed in Theorem \ref{thmalpha:LiftsAndGEquivariantStructure}. 
However, if the SETO $\cC$ has an anomalous symmetry fractionalization class $o_{4}\in H^{4}(G,U(1))$, we have not yet determined whether the phase $\cC_A^{\loc}$ after anyon condensation is anomalous or not \cite{1710.09391}. 
We leave this question for future study.

\begin{thmalpha}
\label{thmalpha:GCrossedBraidedFromGEquivariant}
Suppose $\cF$ is a $G$-crossed braided extension of $\cC$ with $G$-crossed braiding $\beta$, and let $\cF_A$ be the right $A$-modules in $\cF$.
Each $G$-equivariant algebra structure $\lambda$ induces a canonical $G$-crossed braided extension $\cE \subset \cF_A$ of $\cC_A^\loc$ 
such that gauging commutes with condensation.
\end{thmalpha}

$$
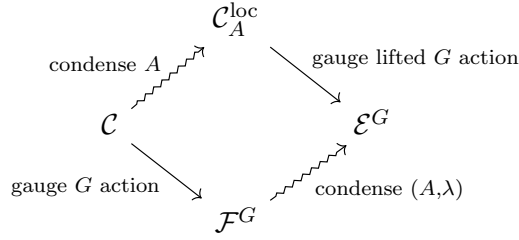
\begin{tikzcd}
  & \cC_A^{\mathrm{loc}} \arrow[rd, "\text{gauge lifted }G\text{ action}"] &  
  \\
  \cC \arrow[rd, "\text{gauge }G\text{ action}"'] \arrow[ru,squiggly, "\text{condense }A"] &  & \cE^G 
  \\
  & \cF^G \arrow[ru,squiggly,"\text{condense }{(A,\lambda)}"']
  &
\end{tikzcd}
$$ 
Indeed, since $A\in \cC = \cF_e$, the $g$-graded component $\cE_g\subset (\cF_A)_g$ of $\cE\subset \cF_A$ consists of the $g$-\emph{local modules} for $A$ in $(\cF,\beta)$, which are the $A$-modules $(M,r)\in (\cF_A)_g$ such that
$$
r=
r\circ \beta_{A,M} \circ (\lambda^g \otimes \id_M)\circ \beta_{M,A}.
$$
We refer the reader to \S\ref{sec:InducedActionForGCrossed} for more details and the proof of Theorem \ref{thmalpha:GCrossedBraidedFromGEquivariant}. 
We remark that we take this theorem as strong evidence that our definition of symmetry being categorically preserved is the strongest possible definition available at the categorical level. 
We also point out \cite[\S5.2]{MR3613518}, where the authors have a symmetry preservation result with a similar flavor to Theorem \ref{thmalpha:GCrossedBraidedFromGEquivariant}, but in a different context.

It is now natural to ask how one can determine if symmetry is categorically preserved or broken for a given categorical action $\uG \to \uEqBr(\cC)$. 
Clearly a first obstruction is that the image of $\uG$ must lie in the full categorical subgroup $\uEqBrStab(A)\subset \uEqBr(\cC)$.
In \S\ref{sec:Obstruction} below, we discuss the obstruction to the existence of a $G$-equivariant algebra structure given a categorical action $\uG \to \uEqBrStab(A)\subset \uEqBr(\cC)$.
We begin by showing there is an exact sequence of groups
\begin{equation}
\label{eq:SESintro}
\begin{tikzcd}
1
\ar[r,""]
&
\Aut_\cC(A) 
\ar[r,"\iota"]
&
\Aut_{\cC^G}(I(A))
\ar[r,"\pi"]
&
G
\ar[r,""]
&
1
\end{tikzcd}
\end{equation}
where $\Aut_\cC(A)$ is the group of algebra automorphisms of $A\in \cC$, $I: \cC \to \cC^G$ is the right adjoint of the forgetful functor $\cC^G \to \cC$, and $\Aut_{\cC^G}(I(A))$ is the group of algebra automorphisms of $I(A)\in \cC^G$.

\begin{thmalpha}
\label{thm:SplittingsCorrespondToGEquivariantAlgebraStructures}
Splittings of the exact sequence \eqref{eq:SESintro} are in bijective correspondence with $G$-equivariant algebra structures $\lambda$ on $A\in \cC$.
\end{thmalpha}

Thus the failure of the short exact sequence \eqref{eq:SESintro} to split forces symmetry to be broken under anyon condensation, independent of the microscopic energetics. 
If the exact sequence does not split for the whole group $G$ but there is a subgroup $H\subset G$ which admits a splitting, then the total symmetry $G$ must break, but is allowed to break down to $H$. 
If $H\subset G$ admits no splitting, however, then the symmetry $G$ can never break down to $H$. 

If the exact sequence (\ref{eq:SESintro}) splits and $G$ symmetry is preserved after condensing $A$, there is also a physical consequence for its inequivalent splittings: they correspond to different gapped $G$-symmetry phases obtained by $A$ condensation. This point will be illustrated by example \S\ref{ex:ToricCodeNoObstruction} using the toric code. 

We emphasize that Theorem \ref{thm:SplittingsCorrespondToGEquivariantAlgebraStructures} applies to a generic SETO, no matter whether it is a regular SETO realized in two-dimensional (2d) lattice models with onsite symmetries, or an anomalous SETO realized on the 2d surface of a three-dimensional SPT phase. 
In other words, the symmetry breaking rules for anyon condensation transitions is valid irrespective of the obstruction $o_{4}\in H^{4}(G,U(1))$ for a 2d SETO.

We examine many examples in detail in \S\ref{sec:Examples}.
We begin with Landau Theory in \S\ref{sec:LandauTheory}, followed by the Toric Code in \S\ref{sec:ToricCode}, and stable $G$-actions where $g=\id_\cC$ for all $g\in G$ in \S\ref{sec:StableActions}.
In \S\ref{subsection:UniversalExample}, we give a universal example using Drinfeld doubles of finite groups, showing that any prescribed exact sequence of finite groups can arise as the second obstruction of categorical symmetry preservation \eqref{eq:SESintro}.
In \S\ref{sec:AutomaticallyPreservedOrBroken}, we include examples where symmetry must be either automatically categorically preserved or broken.

Finally, in \S\ref{sec:AQFT} below, 
we give an application to algebraic quantum field theory.
Given a rational conformal net $\cB$, whose representation category $\Rep(\cB)$ is a UMTC, extensions $\cA\supset \cB$ correspond to \emph{irreducible Q-systems} in $\Rep(\cB)$, a.k.a.\ connected \'{e}tale ${\rm C^*}$ Frobenius algebras $A\in \Rep(\cB)$ \cite{MR1332979,MR3424476}.
We prove in Proposition \ref{prop:CNets} that a global symmetry $G$ of $\cB$ extends to $\cA$ only if the associated categorical action $\uG \to \uEqBr(\cC)$ lies in $\uEqBrStab(A)$.
Moreover, extensions $\cA\supset \cB$ are in bijective correspondence with splittings of the exact sequence \eqref{eq:SESintro}.

\paragraph{Acknowledgements.}
The authors would like to thank the following people for helpful comments and conversations:
Alexei Davydov, 
C\'{e}sar Galindo,
Liang Kong, 
Tian Lan,
Dominic Williamson, 
Yilong Wang,
Zhenghan Wang, and
Xiao-Gang Wen.
Marcel Bischoff was suppported by NSF DMS grant 1700192/1821162.
Corey Jones was supported by NSF DMS grant 1654159.
Yuan-Ming Lu was supported by NSF DMR grant 1653769.
David Penneys was supported by NSF DMS grants 1500387/1655912 and 1654159.

\section{Categorical groups, condensation, and gauging}

In this section, we give the requisite backgound on categorical groups, (de)equivariantization for fusion categories, and condensation and gauging for braided fusion categories.
We refer the reader to \cite{MR3242743} for background on tensor categories, module categories, algebra objects, and module objects.

\subsection{Categorical group actions and levels of symmetry}

In this article, we identify a group $G$ with the \emph{categorical 0-group} which has one object $*$ and $\Hom(*\to *)=G$ with composition given by the group law.

\begin{defn}
A \emph{categorical (1-)group} is a 2-category with one object $*$, every 1-morphism is invertible up to a 2-morphism, and all 2-morphisms are invertible.

Given a categorical group, we obtain a group by \emph{truncation}, where we identify all isomorphic 1-morphisms, and we forget the 2-morphisms.
\end{defn}


\begin{nota}
\label{nota:CategoricalGroup}
We use a sans-serif font with 1 underline to denote a categorical 1-group $\usfG$, and its truncation to a group is denoted without an underline by $\sfG$.
We use the standard font with serifs for a group $G$, and we denote the categorical 1-group obtained by only adding identitiy morphisms at level 2 by $\uG$.
Thus the sans-serif font $\sfG$ should signify to the reader that some information has been lost from $\usfG$, while $G$ has lost no information from $\uG$.
\end{nota}




\begin{defn}
Given a categorical group $\usfG$,
we define $\pi_0(\usfG) := \{*\}$, $\pi_1(\usfG)$ is the group of equivalence classes of automorphisms of $*$, and $\pi_2(\usfG)$ is the group of automorphisms of $\id_*$.
\end{defn}

\begin{ex}[Sinh, {\cite[\S4.2]{MR2664619}}]
Suppose we have a tuple $(H,A,\pi,\omega)$ where $H$ is a group, $A$ is an abelian group, $\pi: H \to \Aut(A)$ is a group homomorphism, and $\omega \in Z^3(H,A,\pi)$.
From this data, we can construct a categorical group $\usfG:=\usfG(H,A,\pi, \omega)$ with 1-morphism set $H$ with composition given by the group law, and 2-morphism set $\Hom(h,k) =\delta_{h=k} A$ for $h,k\in H$.
The homomorphism $\pi$ satisfies
$$
\id_h \otimes a = \pi_h(a)\otimes \id_g
\in \End(h)
\qquad\qquad
\forall h\in H, a\in A,
$$
and the associator 2-morphisms $(h\otimes k)\otimes \ell \to h \otimes (k\otimes \ell)$ for $h,k,\ell\in H$ are determined by the 3-cocycle $\omega$.
This immediately implies $\pi_1(\usfG)=H$ and $\pi_2(\usfG)=A$.
It is straightforward to verify that cohomologous 3-cocycles give equivalent categorical groups. 
Moreover, every categorical group is equivalent to one of this form.
\end{ex}

\begin{defn}
For a semisimple linear category $\cC$,
$\uEq(\cC)$ is the categorical group with $\End(*)$ the set of autoequivalences of $\cC$ and 2-morphisms natural isomorphisms of functors.

For a (semisimple) tensor category $\cC$,
$\uEqTens(\cC)$ is the categorical group with $\End(*)$ the set of tensor autoequivalences of $\cC$ and 2-morphisms monoidal natural isomorphisms of tensor functors.

For a (semisimple) braided tensor category $\cC$,
$\uEqBr(\cC)$ is the categorical group with $\End(*)$ the set of braided tensor autoequivalences of $\cC$ and 2-morphisms monoidal natural isomorphisms of tensor functors.
\end{defn}

The next definition is based on \cite{1410.4540,MR3555361}.

\begin{defn}
\label{defn:LevelsOfSymmetry}
Suppose $\cC$ is a (semisimple) tensor category and $G$ is a group.
There are two levels of symmetry action in this setting:
\begin{enumerate}[(1)]
\item 
A \emph{first level symmetry action} is a group homomorphism $\rho:G \to \EqTens(\cC)$.
\item
A \emph{second level symmetry action}, also called a \emph{categorical $G$-action}, is a monoidal functor $(\underline{\rho},\mu): \uG \to \uEqTens(\cC)$.
Notice that a second level symmetry decategorifies to a first level symmetry.
\end{enumerate}
Observe that there are analogous definitions of first and second level symmetry actions when $\cC$ is just a semisimple category, where the targets are $\Eq(\cC)$ and $\uEq(\cC)$ respectively.

Now suppose $\cC$ is a non-degenerately braided fusion category.
We now have three levels of symmetry action:
\begin{enumerate}[(1)]
\item 
A \emph{first level symmetry action} is a homomorphism $\rho:G\rightarrow \EqBr(\cC)$.

\item
A \emph{second level symmetry action}, also called a \emph{categorical $G$-action},
is a monoidal functor $\underline{\rho}:\uG\rightarrow \uEqBr(\mathcal{C})$.
Again, a second level symmetry decategorifies to a first level symmetry. 

\item
A \emph{third level symmetry action} is a $G$-crossed braided extension of $\cC$. 
The restriction of the $G$-action to a monoidal functor $\uG\rightarrow \uEqBr(\cC)$ is a second level symmetry.
\end{enumerate}
\end{defn}

\subsection{(De)equivariantization, condensation, and gauging}

We now rapidly recall the notions of (de)equivariantization, condensation, and gauging.
A general reference for this material is the book \cite{MR3242743}.
Other references include \cite{MR1936496,MR2609644,MR3039775,MR3555361}.
For this section, $\cC$ is a fusion category.

\begin{nota}
\label{nota:CategoricalAction}
Suppose $(\underline{\rho},\mu):\uG\rightarrow \uEqTens(\mathcal{C})$ is a categorical action,
where the tensorator $\mu = \{\mu_{g,h}: \underline{\rho}(g) \circ \underline{\rho}(h) \to \underline{\rho}(gh)\}_{g,h\in G}$ is a family of monoidal natural isomorphisms satisfying associativity and unitality axioms.
To ease the notation, we still write $g$ for $\underline{\rho}(g)$.
We denote the tensorator of $g$ by $\psi^g = \{\psi^g_{a,b}: g(a)\otimes g(b) \xrightarrow{\sim} g(a\otimes b)\}_{a,b\in\cC}$.

By \cite[Thm.~1.1]{MR3671186}, one may assume that the action is \emph{strict}, so that $g\circ h = gh$ for all $g,h\in G$.
However, for the sake of generality, we will only assume \emph{strict unitality} of the action:
\begin{itemize}
\item 
Each monoidal functor $(g, \psi^g)$ is \emph{unital} \cite[Prop.~3.1]{MR3671186}, i.e., for all $g\in G$, $g(1_\cC)= 1_\cC$ and $g(\id_{1_\cC}) = \id_{1_\cC}$, and
\item
$e=\id_\cC$, $e\circ g= g\circ e =g$ and $\mu_{g,e} = \mu_{e,g} = \id_g$ for all $g\in G$.
\end{itemize}
\end{nota}

\begin{defn}
\label{defn:Equivariantization}
A $G$-\emph{equivariant object} is a pair $(X,\lambda)$ where $X\in \cC$ and $\lambda = \{\lambda^{g}:g(X)\rightarrow X \}_{g\in G}$ is a family of isomorphisms such that the following diagram commutes for all $g,h\in G$:
\begin{equation}
\label{eq:EquivariantObject}
\begin{tikzcd}
g(h(X)) \arrow{r}{g(\lambda^{h})} \arrow{d}[swap]{\mu^{X}_{g,h}} 
& 
g(X) \arrow{d}{\lambda^{g}} 
\\
gh(X) \arrow{r}{\lambda^{gh}} 
& 
X.
\end{tikzcd}
\end{equation}
Given $G$-equivariant objects $(X,\lambda), (Y,\kappa)$, 
we call a morphism $f\in \cC(X\to Y)$ a $G$-\emph{equivariant morphism} if the following diagram commutes for all $g\in G$:
\begin{equation}
\label{eq:EquivariantMorphism}
\begin{tikzcd}
g(X) \arrow{r}{\lambda^{g}} \arrow{d}[swap]{g(f)} 
& 
X \arrow{d}{f} 
\\
g(Y) \arrow{r}{\kappa^{g}} 
& 
Y.
\end{tikzcd}
\end{equation}
The \emph{equivariantization} $\cC^G$ is the category whose objects are $G$-equivariant objects and whose morphisms are $G$-equivariant morphisms.
The tensor product in $\cC^G$ is given by 
\begin{equation}
\label{eq:TensorProductOnEquivariantization}
(X,\lambda)\otimes (Y,\kappa) := (X\otimes Y, (\lambda^g \otimes \kappa^g)\circ(\psi^g_{X,Y})^{-1})
\end{equation}
and the unit object is $(1_\cC, \id_{1_\cC})$.
\end{defn}

\begin{rem}
\label{rem:StrictlyUnitalEquivariantStructure}
Observe that when the $G$-action is strictly unital, the commutativity of \eqref{eq:EquivariantObject} with $g=h=e$ shows that any $G$-equivariant object $(X,\lambda)$ must have $\lambda^e = \id_X$.
\end{rem}

The inverse process to equivariantization is \emph{de-equivariantization}. 

\begin{defn}
\label{defn:De-equivariantization}
Suppose $\iota:\Rep(G)\to\cZ(\cC)$ is a fully faithful braided tensor functor such that the composite $F\circ \iota : \Rep(G) \to \cC$ is still fully faithful, where $F:\cZ(\cC)\to\cC$ is the forgetful functor.
(Such an inclusion $\Rep(G) \subset \cZ(\cC)$ is called a \emph{Tannakian} subcategory.)
Let $\cO(G)\in \Rep(G)$ denote the algebra object of functions $G\to \bbC$ whose multiplication is given by $\chi_g \cdot \chi_h = \delta_{g=h} \chi_g$ where $\chi_g(h) = \delta_{h=g}$ for $g,h\in G$.
Then $\iota(\cO(G))$ is an \'{e}tale algebra object in $\cC$ whose
category of right modules $\cC_G:=\cC_{\iota(\cO(G))}$
is a fusion category, called the \emph{de-equivariantization} of
$\cC$.

More generally, given a separable algebra object $A\in \cC$ which lifts to a commutative (and thus \'{e}tale) algebra in the Drinfeld center $Z(\cC)$, the category $\cC_A$ of right $A$-modules is a tensor category which is called the de-equivariantization of $\cC$ by $A$.
In more detail, if $e_A=\{e_{A,c}: A\otimes c \xrightarrow{\sim} c\otimes A\}_{c\in \cC}$ is a half-braiding for $A$ such that $m,i$ are morphisms in $Z(\cC)$, we define the tensor product of right $A$-modules $(M,r_M)$ and $(N,r_N)$ as the image of the \emph{separability idempotent}
\begin{equation}
\label{eq:SeparabilityIdempotent}
p_{M,N}
:=
(r_M \otimes r_N)
\circ
(\id_{M\otimes A} \otimes e_{A,N})
\circ
(\id_M \otimes (s\circ i) \otimes \id_N)
\in
\cC_A(M\otimes N \to M\otimes N)
\end{equation}
where $s\in\Hom_{A-A}(A \to A\otimes A)$ is a \emph{splitting} such that $m\circ s = \id_A$.
When $A$ is connected, $s$ is unique, since 
$\Hom_{A-A}(A \to A\otimes A) \cong \cC(A \to 1) = \bbC i$ similar to \cite[Fig.~4]{MR1936496}.
\end{defn}

\begin{fact}
For fusion categories, the maps $\cC \mapsto \cC^G$ and $\cD \mapsto \cD_G$ are mutually inverse up to equivalence; we refer the reader to \cite[Rem.~8.23.5]{MR3242743} for more details.
\end{fact}

In the presence of a braiding $\beta$ on $\cC$, given an algebra object $A\in \cC$, we get a canonical lift of $A$ to $Z(\cC)$ using the half-braiding $\beta_A$.
We can ask whether the de-equivariantization $\cC_A$ carries a canonical braiding.
In general, this is not the case, so we pass to the subcategory $\cC_A^\loc$ of \emph{local/dyslectic} right modules $(M,r)\in \cC_A$ which satisfy
$r \circ \beta_{A,M}\circ \beta_{M,A} = r$.
The braiding $\beta$ defines a canonical braiding on $\cC_A^\loc$.
The braided fusion category $\cC_A^\loc$ is called the \emph{condensation} of $\cC$ by $A$.

\begin{ex} 
Let $\cC$ be a braided fusion category. 
An invertible object $g\in\cC$ is called a \emph{boson} or \emph{simple-current} 
if $\beta_{g,g}=\id_{g\otimes g}$.
If $B\subset \Inv(\cC)$ is a subgroup consisting of bosons, then the full subcategory of $\cC$ generated by $B$ is braided equivalent to $\Rep(\widehat{B})$, where $\widehat{B}$ is the dual group of $B$.
We call $\cO(\widehat{B})$ the \'{e}tale algebra induced by the group of bosons $B$.
The condensation $\cC_{\cO(\widehat{B})}^\loc$ is also referred to as the braided tensor category obtained by \emph{condensing the bosons} $B$.
\end{ex}

An inverse process to condensing a connected \'{e}tale algebra of the form $\cO(G)\in \Rep(G)\subset \cC$ is given by \emph{gauging} \cite{MR3555361,1410.4540}, which is a two step process consisting of:
\begin{enumerate}[(1)]
\item finding a $G$-crossed braided extension $\cF$ of $\cC$,
which carries a canonical categorical action $\underline{G}\to \uEqTens(\cF)$, and 
  \item taking the equivariantization $\cF^G$.
\end{enumerate}

\begin{defn}[{\cite[\S8.24]{MR3242743}}]
A $G$-crossed braided extension of $\cC$ based on a categorical action $(\underline{\rho},\mu):\uG \to \uEqBr(\cC)$ is a fusion category $\cF$
equipped with the following structure:
\begin{itemize}
\item 
$\cF$ is a faithfully $G$-graded extension of $\cC$, i.e.\ $\cF=\bigoplus \cF_g$
  with $\cF_e=\cC$,
  $X_g\otimes Y_h\in\cF_{gh}$ for every $g,h\in G$, and
    $X_g\in\cF_g$, and $Y_h\in\cF_h$.
\item 
There is an action $(\underline{\rho}^\cF,\mu^\cF):\uG\to \uEqTens(\cF)$ extending the action on $\cF_e=\cC$ such that $g(\cF_h)\subset \cF_{ghg^{-1}}$ for all $g,h\in G$.
\item 
There is a family of natural isomorphisms
    \begin{align}
      \beta_{X,Y} &\colon X\otimes Y\to g(Y)\otimes X\,, 
      &g\in G, X\in\cF_g, Y\in \cF
    \end{align}
     called the $G$-\emph{crossed braiding}, which extends the braiding $\beta$ on $\cF_e=\cC$.
\end{itemize}
Moreover, this data must satisfy the following coherence identities:
\begin{align*}
((\mu_{g,h}^\cF)_Y\otimes \id_{g(X)})
\circ
(\psi^g_{g(Y),X})^{-1}
\circ 
g(\beta_{X,Y})
\circ
\psi^g_{X,Y}
&=
((\mu^\cF_{ghg^{-1},g})_Y\otimes \id_{g(X)})
\circ
\beta_{g(X),g(Y)}
&&\forall X\in \cF_h
\\
(\psi^g_{Y,Z}\otimes \id_X)\circ
(\id_{g(Y)}\otimes \beta_{X,Z})
\circ
(\beta_{X,Y}\otimes \id_Z)
&=
\beta_{X,Y\otimes Z}
&
&\forall X\in \cF_g
\\
((\mu^\cF_{g,h})_Z\otimes \id_{X\otimes Y})
\circ
(\beta_{X,h(Z)}\otimes \id_Y)
\circ
(\id_X\otimes \beta_{Y,Z})
&=
\beta_{X\otimes Y, Z}
&&
\hspace{-1.3cm}
\forall X\in \cC_g, Y\in \cC_h
\end{align*}
\end{defn}

\begin{fact}
When the $G$-action on $\cC$ is strict, which we may assume by \cite[Thm.~1.1]{MR3671186}, every $G$-crossed braided extension $\cF$ of $\cC$ is equivalent to a strict $G$-crossed braided extension of $\cC$ by \cite[Thm.~5.6]{MR3671186}.
That is, if we only consider strict $G$-actions, we do not lose any $G$-crossed braided extensions.
\end{fact}

\begin{fact}
For non-degenerately braided fusion categories, condensing $\cO(G)$ and gauging a second level categorical $G$-symmetry (taking the equivariantization of a $G$-crossed braided extension) are mutually inverse; we refer the reader to \cite[\S4]{MR2609644} and \cite{MR3555361} for more details.
\end{fact}

\begin{rem}
As mentioned above, we can condense any \'{e}tale algebra in a nondegenerately braided fusion category, not just one of the form $\cO(G)$.
It is an important open question to find the inverse process to this more general condensation.
The recent article \cite{1809.00245} provides an interesting step in this direction.
\end{rem}

\section{Obstruction for equivariant algebras}
\label{sec:Obstruction}

Suppose $\cC$ is a tensor category and $(\underline{\rho},\mu):\underline{G}\rightarrow \uEqTens(\mathcal{C})$
is a categorical action.
We assume Notation \ref{nota:CategoricalAction} to ease the notation.
In this section, we study a connected $G$-\emph{stable} algebra object in $\cC$, i.e., a connected algebra object $(A,m,i)\in \cC$ such that $g(A)\cong A$ as algebras for all $g\in G$.

Our first task is to determine when $A$ has the structure of a $G$-\emph{equivariant} algebra object, i.e. an algebra in the equivariantization $\mathcal{C}^{G}$, with the same multiplication map $m$ and unit $i$.
It is easy to show that this is equivalent to having a $G$-equivaraint structure $\lambda=\{\lambda^{g}: g(A)\rightarrow A\}_{g\in G}$ on the object $A$ such that each $\lambda^{g}$ is an algebra isomorphism (where the algebra $g(A)$ has multiplication and unit $g(m)\circ \mu^A_{g,g}, g(i)$ respectively).
That is, we must supply a family of algebra isomorphisms $\lambda=\{\lambda^{g}:g(A)\rightarrow A \}_{g\in G}$ such that \eqref{eq:EquivariantObject} holds.
We prove that the \emph{obstruction} to the existence of such a family $\lambda$ is a certain short exact sequence of groups, and the data of a splitting for this sequence is equivalent to the data of a $G$-equivariant algebra structure.
If $\Aut_\cC(A)$ is abelian, this is equivalent to the vanishing of a certain 2-cocycle in $H^2(G, \Aut_\cC(A))$ (see \S\ref{sec:exactsequence} below).

\subsection{The exact sequence associated to a \texorpdfstring{$G$}{G}-stable algebra object}

To construct our exact sequence for our $G$-stable algebra $A$, we construct a new larger $G$-equivariant algebra object.
Recall that there is a forgetful tensor functor $F: \cC^G \to \cC$ which forgets the $G$-equivariant structure.
Let $I: \cC \to \cC^G$ be the right adjoint of $F$, which we think of as an `induction' functor.
Since $F$ is a tensor functor, $I$ can be canonically endowed with the structure of a lax monoidal functor \cite{MR0360749}, which maps algebra objects to algebra objects.
We now give the explicit description of the algebra $I(A)$.

\begin{defn}
\label{defn:AlgebraI(A)}
We define the object $I(A):=\bigoplus_{g\in G} g(A)$.
The multiplication morphism $n\in \cC^G(I(A) \otimes I(A) \to I(A))$  is given on components by
$$
n^{k}_{g,h}
:=
\delta_{g=h}\delta_{g=k}\cdot  g(m)\circ \psi^g_{A,A}
:
g(A)\otimes h(A)\rightarrow k(A)
$$
The unit morphism $j \in \cC^G(1 \to I(A))$ is given on components by
$$
j_g
:=g(i)
: 1 \to g(A).
$$
It is straightforward to verify that the unit and multiplication
are unital and associative.

The $G$-equivariant structure $\nu=\{\nu^g : g(I(A)) \to I(A)\}_{g\in G}$ is given on components by
$$
\nu^{g}_h
:=
\mu^{A}_{g,h}
:
g(h(A)) \to (gh)(A).
$$
Again, it is straightforward to verify that $n$ and $j$ are $G$-equivariant maps, and thus give well-defined morphisms in $\cC^G$.
\end{defn}

We now construct a short exact sequence
\begin{equation}
\label{eq:ShortExactSequence}
\begin{tikzcd}
1
\ar[r,""]
&
\Aut_\cC(A) 
\ar[r,"\iota"]
&
\Aut_{\cC^G}(I(A))
\ar[r,"\pi"]
&
G
\ar[r,""]
&
1.
\end{tikzcd}
\end{equation}
We first analyze $\Aut_{\cC^G}(I(A))$.

\begin{lem}
\label{lem:UniqueNonzeroComponent}
Suppose $f\in \Aut_{\cC^G}(I(A))$.
\begin{enumerate}[(1)]
\item 
For $h,k\in G$, $f_{h,k}: k(A) \to h(A)$ is equal to $\mu^A_{h,e}\circ h(f_{e,h^{-1}k})\circ (\mu^A_{h,h^{-1}k})^{-1}$.
Hence $f$ is completely determined by its components $f_{e,g}: g(A) \to A$.
\item
There is a unique $g\in G$ such that $f_{e,g}\neq 0$, and $f_{e,g}: g(A) \to A$ is an algebra isomorphism.
\item
For every $h\in G$, there are unique $g,k\in G$ such that $f_{g,h}\neq 0\neq f_{h,k}$.
\end{enumerate}
\end{lem}
\begin{proof}
Note (1) follows immediately since $f: I(A) \to I(A)$ is $G$-equivariant from \eqref{eq:EquivariantMorphism}.

To prove (2), we first note that for each $g\in G$, there is a scalar $\gamma_g \in \bbC$ such that $f_{e,g}\circ g(i) = \gamma_g \cdot i \in \cC(1\to A) = \bbC \cdot i$.
Looking at the $e$-component of the unitality axiom for $(I(A),n,j)$ gives the identity
$$
\id_{A}
=
\sum_{h\in G} m\circ\psi^{A,A}_e \circ ((f_{e,h}\circ h(i))\otimes \id_{A})
=
\left(\sum_{h\in G} \gamma_h \right) m\circ (i \otimes \id_A)\textbf{}
=
\left(\sum_{h\in G}\gamma_h\right) \id_A,
$$
which immediately implies $\sum_h \gamma_h = 1$.
Fix $h\in G$ such that $\gamma_h \neq 0$.
For $g\neq h$, looking at the component $n_{h,g}^e: h(A)\otimes g(A) \to A$ yields the identity
$$
m \circ \psi^{A,A}_{e} \circ (f_{e,h}\otimes f_{e,g})
=
\delta_{h=g}
f_{e,g}\circ g(m)\circ\psi^{A,A}_g
=
0.
$$
Precomposing with $h(i)\otimes \id_{g(A)}$ yields
$$
0
=
m \circ ((f_{e,h}\circ h(i)) \otimes f_{e,g})
=
\gamma_h
\cdot
m\circ (i\otimes f_{e,g})
=
\gamma_h\cdot
f_{e,g}.
$$
Since $\gamma_h \neq 0$, we conclude $f_{e,g}=0$ whenever $g\neq h$, proving (2).
Notice this also proves $\gamma_h = 1$.
That $f_{e,g}: g(A) \to A$ is an algebra isomorphism follows immediately by looking at components as above.

Now (3) follows immediately by (1) and (2).
\end{proof}

\begin{prop}
\label{prop:InjectiveGroupHom}
Defining $\iota:\Aut_\cC(A) \to \Aut_{\cC^G}(I(A))$ by
$\iota(f)_{h,g}:=\delta_{h=g}\cdot g(f)$ gives a well-defined injective group homomorphism.
\end{prop}
\begin{proof}
Given $f\in \Aut_\cC(A)$, define the morphism $\iota(f)_{g,h}:=\delta_{g=h} \cdot g(f)$. 
It is easy to verify this is an algebra automorphism of $I(A)$. Furthermore, since $\iota(f)_{e,e}=f$, this is clearly injective. 
\end{proof}

\begin{prop}
\label{prop:DefineSurjectiveGroupHom}
Defining $\pi: \Aut_{\cC^{G}}(I(A)) \to G$ by $f\mapsto g$ where $g\in G$ is the unique element such that $f_{e,g}\ne 0$ gives a group homomorphism.
\end{prop}
\begin{proof}
Suppose $f^1,f^2\in \Aut_{\cC^G}(I(A))$, and consider $f^1\circ f^2$. 
Then $\pi(f^1\circ f^2)$ is the unique element $g\in G$ such that $(f^1\circ f^2)_{e, g}\neq 0$. 
We calculate that
$$
(f^1\circ f^2)_{e, g}
=
\sum_{h\in G}f^1_{e,h}\circ f^2_{h,g}
=
f^1_{e,\pi(f^1)}\circ f^2_{\pi(f^1),g}.
$$
Now notice that $f^2_{\pi(f^1),g}\neq 0$ if and only if $f^2_{e, \iota(f^1)^{-1}g}\neq 0$.
Hence $\pi(f^1)^{-1}g=\pi(f^2)$, which immediately implies
$\pi(f^1\circ f^2)=g=\pi(f^1)\cdot \pi(f^2)$.
\end{proof}

\begin{lem}
\label{lem:SurjectiveGroupHom}
Fix $g\in G$.
Given an algebra isomorphism $\lambda^g : g(A) \to A$, there is a unique $f\in \Aut_{\cC^G}(I(A))$ such that $f_{e,g} = \lambda^g$.
The assignment $\lambda^g \mapsto f$ is a bijection between the set of algebra isomorphisms $g(A)\to A$ and the pre-image $\pi^{-1}(g)$.
\end{lem}
\begin{proof}
Let $g\in G$, and let $\lambda^{g}: g(A)\mapsto A$ be an algebra isomorphism. 
If there is an $f\in \Aut_{\cC^G}(I(A))$ such that $f_{e,g}=  \lambda^{g} $, then by Lemma \ref{lem:UniqueNonzeroComponent}, for arbitrary $h,k\in G$, we must have 
\begin{equation}\label{section formula}
f_{h,k}:=
\delta_{k=hg}\cdot
h(\lambda^{g})\circ (\mu^A_{h,h^{-1}k})^{-1}.
\end{equation}
(Notice here that $\mu_{h,e}^A=\id_{h(A)}$ as discussed in Notation \ref{nota:CategoricalAction}.)
We claim defining $f$ in this way yields an algebra automorphism of $I(A)$.
To see that $f$ is compatible with multiplication map $n$ for $I(A)$, we compute
\begin{align*}
h(m)\circ \psi^{A,A}_{h}
&\circ 
(f_{h,k}\otimes f_{h,l})
\\&=
\delta_{k=l=hg}\,\, 
h(m)\circ \psi^{A,A}_{h}
\circ 
\left( h(\lambda^{g})
\circ 
(\mu^A_{h,h^{-1}k})^{-1} \otimes h( \lambda^{g})
\circ 
(\mu^A_{h,h^{-1}k})^{-1}\right) 
\\&=
\delta_{k=l=hg}\,\, 
h(m)
\circ 
\psi^{A,A}_{h}
\circ 
\left( h(\lambda^{g})\otimes h(\lambda^{g})\right)
\circ 
\left((\mu^A_{h,h^{-1}k})^{-1} \otimes (\mu^A_{h,h^{-1}k})^{-1}\right)
\\&=
\delta_{k=l=hg}\,\, 
h(m\circ ( \lambda^{g} \otimes \lambda^{g}))
\circ 
\psi^{g(A),g(A)}_{h}
\circ 
\left((\mu^A_{h,g})^{-1} \otimes (\mu^A_{h,g})^{-1}\right)
\\&=
\delta_{k=l=hg}\,\, 
h( \lambda^{g}\circ g(m))\circ h(\psi^{A,A}_{g})
\circ 
\psi^{g(A),g(A)}_{h}
\circ  
\left((\mu^A_{h,g})^{-1} \otimes (\mu^A_{h,g})^{-1}\right)
\\&=
\delta_{k=l}\delta_{k=hg}\,\, 
h(\lambda^{g})
\circ 
h(g(m))
\circ 
\psi^{A,A}_{h\circ g}
\circ
\left((\mu^A_{h,g})^{-1} \otimes (\mu^A_{h,g})^{-1}\right)
\\&=
\delta_{k=l}\delta_{k=hg}\,\, 
h(\lambda^{g})
\circ 
(\mu^{A}_{h,g})^{-1}
\circ 
hg(m)
\circ 
\psi^{A,A}_{hg}
\\&=
\delta_{k=l}\,\,
f_{h,k}
\circ
k(m)
\circ
\psi^{A,A}_{k}
\end{align*}
We omit the easier proof that $f$ is compatible with the unit map $j$ of $I(A)$.

Now since $f$ is completely determined by $f_{e,g}$, and $\pi(f)=g$ by definition of $\pi$, we see that the assignment $\lambda^g \mapsto f$ is a bijection.
\end{proof}

\begin{thm}
The sequence \eqref{eq:ShortExactSequence} is exact.
\end{thm}
\begin{proof}
By Proposition \ref{prop:InjectiveGroupHom}, $\iota$ is injective.
By Lemma \ref{lem:SurjectiveGroupHom}, $\pi$ is surjective.
It remains to show that $\ker(\pi)=\iota(\Aut_{\mathcal{C}}(A))$. 
Suppose $f\in \ker(\pi)$. 
By Lemma \ref{lem:UniqueNonzeroComponent}, $f$ is determined by translations of the algebra automorphism $f_{e,e}: A \to A$ in $\Aut_\cC(A)$. 
Thus $f=\iota(f_{e,e})$, and we are finished.
\end{proof}

\subsection{Exact sequences of groups and cohomology}\label{sec:exactsequence}

We now include a short aside to discuss splitting of exact sequences and its relation to cohomology.
All this material is well-known, and we refer the reader to \cite{MR672956} for more details.

Consider a short exact sequence of groups
\begin{equation}
\label{eq:ShortExactSequenceOfGroups}
\begin{tikzcd}
1 \arrow[r] 
& 
N \arrow[r, "\iota"] 
& 
E \arrow[r, "\pi"] 
& 
G \arrow[r]
& 
1.
\end{tikzcd}
\end{equation}
Such a short exact sequence induces a homomorphism $\alpha: G\rightarrow \Out(N)$. 
Now starting with a homomorphism $\alpha: G \to \Out(N)$, it is natural to ask what are the possible extensions of $G$ by $N$ that induce $\alpha$. 
There may be no such extensions, but if one exists, the set of all possible extensions forms a torsor over the group $H^{2}(G, Z(N),\alpha)$.

Recall that a \emph{splitting} for the exact sequence \eqref{eq:ShortExactSequenceOfGroups}
\begin{equation}
\label{eq:SplittingsOfSES}
\begin{tikzcd}
1 \arrow[r] 
& 
N \arrow[r, "\iota"] 
& 
E \arrow[r, "\pi"] 
& 
G \arrow[r] \arrow[l, "\sigma", dashed, bend left] 
& 
1
\end{tikzcd}
\end{equation}
is a group homomorphism $\sigma: G \to E$ such that $\pi \circ \sigma = \id_G$, which is equivalent to the existence of a group isomorphism $E \cong N \rtimes G$ for some action of $G$ on $N$.
Given a splitting $\sigma: G \to E$ and $n\in N$, we get another splitting by defining $\sigma_n(g):=n \sigma(g)n^{-1}$.
Hence $N$ acts on the set of splittings \eqref{eq:SplittingsOfSES} by conjugation.
We call two splittings \emph{equivalent} if they are in the same $N$-conjugacy class.

When $N$ is abelian, $\Out(N) = \Aut(N)$, and thus there exists a canonical extension associated to any $\alpha: G \to \Aut(N)$, namely the semi-direct product.
This canonical extension is a canonical basepoint for our torsor, and thus extensions inducing $\alpha$ are classified by the \emph{group} $H^{2}(G, N,\alpha)$. 
Moreover, the extension is split if and only if the assocaited 2-cocycle $\omega \in H^2(G, N,\alpha)$ is trivial, and splittings exactly correspond to functions $\gamma: G \to N$ satisfying $\gamma(g) \alpha_{g}(\gamma(h))=\gamma(gh)$, i.e., $\gamma$ is a $1$-cocycle in $Z^1(G,N,\alpha)$.
Two splittings are then equivalent if and only if they are cohomologous
\cite[Ch.~IV, Prop.~2.3]{MR672956}.

\subsection{The splitting obstruction and \texorpdfstring{$G$}{G}-equivariance}

We now prove our main theorem for this section.

\begin{thm}
The data of a splitting for the exact sequence \eqref{eq:ShortExactSequence} is equivalent to the data of a $G$-equivariant algebra structure for $A$.
\end{thm}
\begin{proof}
Suppose $\sigma: G \to \Aut_{\cC^G}(I(A))$ is a section.
Then for each $g\in G$, $\sigma(g) \in \pi^{-1}(g)$, and $\sigma(g)$ is completely determined by the algebra isomorphism $\lambda^g:=\sigma(g)_{e,g} : g(A) \to A$ by Lemma \ref{lem:SurjectiveGroupHom}.

We now look at the equation $\sigma(g)\circ \sigma(h) = \sigma(gh) \in \Aut_{\cC^G}(I(A))$, which in components is
$\sum_{j\in G} \sigma(g)_{i,j} \circ \sigma(h)_{j,k} = \sigma(gh)_{i,k}$.
By the definition of the multiplication morphism $n$ for $I(A)$, both sides are zero unless $i=kh^{-1}g^{-1}$, and thus this equation is equivalent to
\begin{equation}
\label{eq:EquivariantAlgebraFromSplitting1}
\sigma(g)_{kh^{-1}g^{-1},kh^{-1}}
\circ 
\sigma(h)_{kh^{-1},k}
=
\sigma(gh)_{kh^{-1}g^{-1},k}.
\end{equation}
Using the formula \eqref{section formula}, we can rewrite
\begin{align*}
\sigma(g)_{kh^{-1}g^{-1},k}
&=
kh^{-1}g^{-1}(\lambda^{g})\circ (\mu^A_{kh^{-1}g^{-1},g})^{-1}
\\
\sigma(h)_{kh^{-1},k}
&=
kh^{-1}(\lambda^{h})\circ (\mu^A_{kh^{-1},h})^{-1}
\\
\sigma(gh)_{kh^{-1}g^{-1},k}
&=
kh^{-1}g^{-1}(\lambda^{gh})\circ (\mu^A_{kh^{-1}g^{-1},gh})^{-1}.
\end{align*}
Thus \eqref{eq:EquivariantAlgebraFromSplitting1} is equivalent to the following equation in terms of $\lambda=\{\lambda^g: g(A) \to A\}_{g\in G}$:
\begin{equation}
\label{eq:EquivariantAlgebraFromSplitting2}
kh^{-1}g^{-1}(\lambda^{g})\circ (\mu^A_{kh^{-1}g^{-1},g})^{-1}\circ kh^{-1}(\lambda^{h})\circ (\mu^A_{kh^{-1},h})^{-1}
= 
kh^{-1}g^{-1}(\lambda^{gh})\circ (\mu^A_{kh^{-1}g^{-1},gh})^{-1}.
\end{equation}

We claim that \eqref{eq:EquivariantAlgebraFromSplitting2} holds for all $g,h,k\in G$ if and only if $(A,m,i,\lambda)$ is a $G$-equivariant algebra.
For the forward direction, if \eqref{eq:EquivariantAlgebraFromSplitting2} holds, then setting $h=k$, we obtain for all $g,h\in G$,
$$
g^{-1}(\lambda^{g})\circ (\mu^A_{g^{-1},g})^{-1}\circ \lambda^{h}
= 
g^{-1}(\lambda^{gh})\circ (\mu^A_{g^{-1},gh})^{-1}.
$$
Now applying $g$ to both sides, post-composing with $\mu^A_{g,g^{-1}}$, and simplifying gives the $G$-equivariance condition \eqref{eq:EquivariantObject}.
Conversely, if $\lambda_g:g(A) \to A$ are a family of algebra isomorphisms making $(A,m,i,\lambda)$ a $G$-equivariant algebra,
we apply $kh^{-1}g^{-1}$ to \eqref{eq:EquivariantObject} and pre-compose with $(\mu^A_{kh^{-1}g^{-1},g})^{-1}$.
Again, simplifying gives \eqref{eq:EquivariantAlgebraFromSplitting2}.

We have thus proven that our section $\sigma$ is a splitting if and only if $\lambda^g := \sigma(g)_{e,g}: g(A) \to A$ defines a $G$-equivariant algebra structure for $A$.
\end{proof}

\begin{rem}
Viewing the splitting of the exact sequence \eqref{eq:ShortExactSequence} as an obstruction bears many similarities to the results of \cite{MR2927179}. 
There, the authors investigate module categories over $G$-graded fusion categories, and they quickly reduce their problem to studying $G$-equivariant module categories. 
The 2-category of modules, module functors, and natural transformations is 2-equivalent to the 2-category of algebras, bimodules and intertwiners, so unsurprisingly, they obtain an obstruction for $G$-equivariant module categories which is the splitting of their short exact sequence (8.1). 

Our situation differs as a truncation of theirs in the following sense. 
Since we are motivated by the physics of anyon condensation, we are effectively considering \textit{module tensor categories} with \emph{basepoints} \cite{MR3578212} over the non-degenerate braided fusion category $\cC$. 
Indeed, given a connected \'{e}tale algebra $A\in \cC$, we get a fully faithful braided tensor functor
$\Phi^{\scriptscriptstyle\cZ}: \cC \to \cZ(\cC_A)\cong \cC \boxtimes \overline{\cC_A^{\loc}}$,  where the bar denotes taking the opposite braiding \cite[Cor.~3.30]{MR3039775}.
The functor $\Phi^{\scriptscriptstyle\cZ}$ equips the tensor category $\cC_A$ with the structure of a module tensor category over $\cC$ with simple basepoint $A\in \cC_A$.

Now pointed module tensor functors between module tensor categories with basepoint do not have an additional layer of structure, as there is at most one natural transformation between any two which is necessarily a natural isomorphism \cite[Lem.~3.5]{1607.06041}.
Thus we may safely truncate to the $1$-category of connected \'{e}tale algebras and algebra morphisms without losing any information. 

As a 1-category has one fewer layer of structure, we have one fewer obstruction. 
However, our obstruction is not simply a special case of theirs, since the group of automorphisms of an algebra (on which our sequence is based) and the group of equivalence classes of invertible bimodules (on which their exact sequence is based) are not the same in general.
\end{rem}

\section{Equivariant algebras and induced categorical actions}

In this section, we show how given a categorical action of $\uG$ on $\cC$ together with a $G$-equivariant structure $\lambda$ on an algebra $A\in \cC$, we get a canonical induced categorical action on the de-equivariantized or condensed theory.
We do this in three different stages.

First, in \S\ref{sec:InducedActionForCommutativeCentralAlgebra} below, we consider a fusion category $\cC$, a connected separable algebra object $A\in \cC$ which lifts to a commutative (and thus \'{e}tale) algebra object $(A,e_A)\in \cZ(\cC)$, and a \emph{central} $G$-equivariant structure $\lambda$ on $A$, i.e., $\lambda^g \in \cZ(\cC)(g(A) \to A)$ for all $g\in G$.
We then show we get a canonical induced categorical action $\uG \to \uEqTens(\cC_A)$.

Second, in \S\ref{sec:InducedActionForEtaleAlgebra} below, we consider a non-degenerately braided fusion category $\cC$, a connected \'{e}tale algebra object $A\in \cC$, and a $G$-equivariant structure $\lambda$ on $A$.
We then show we get a canonical induced categorical action $\uG \to \uEqBr(\cC_A^\loc)$.

Third, in \S\ref{sec:InducedActionForGCrossed}, we combine the above two situations and consider the case of a connected \'{e}tale algebra in a non-degenerately braided fusion category $\cC$, together with an action $\uG\rightarrow \uEqBr(\cC)$ and a compatible $G$-crossed braided extension $\cC\subseteq \cF$. 
We apply the earlier results to show that for each $G$-equivariant algebra structure $(A, \lambda)$, there exists a canonical $G$-crossed braided extension $\cC^{\loc}_{A}\subseteq \cE$, compatible with the action $\uG \to \uEqBr(\cC_A^\loc)$ constructed in \S\ref{sec:InducedActionForEtaleAlgebra}. 
We also demonstrate the compatibility of $\cE$ and $\cF$ by showing that gauging and condensation commute (see Theorem \ref{thm:GaugingCommutes}).

\subsection{The induced action on \texorpdfstring{$\cC_A$}{Mod-A}}
\label{sec:InducedActionForCommutativeCentralAlgebra}

We now show that given a separable algebra object $A\in \cC$ which lifts to a commutative (and thus \'{e}tale) algebra object $(A,e_A)\in \cZ(\cC)$, together with a \emph{central} $G$-equivariant structure $\lambda$ on $A$, 
a categorical actiom
$(\underline{\rho},\mu):\underline{G}\rightarrow \uEqTens(\mathcal{C})$
induces a canonical categorical action
$(\underline{\rho}_A,\mu):\underline{G}\rightarrow \uEqTens(\mathcal{C}_A)$.
First, we define a number of categorical groups and describe how they are related.

\begin{defn}
\label{defn:ExamplesOfCategoricalGroups}
We define the following groups and categorical groups:
\begin{itemize}
\item 
$\uEqStab(A)\subset \uEqTens(\cC)$ is the full categorical subgroup whose 1-morphisms are $\alpha\in \uEqTens(\cC)$ such that $\alpha(A)\cong A$ as algebras
with composition of 1-morphisms from $\uEqTens(\cC)$,
and whose 2-morphisms are monoidal natural isomorphisms.
\item
$\EqStab(A)$ is the group truncation (or $\pi_{1}$) of $\uEqStab(A)$, whose elements are equivalence classes of $\alpha \in \EqTens(\cC)$ such that $\alpha(A) \cong A$.
\item
$\Aut_\cC(A)$ is the group of algebra isomorphisms $\lambda: A\to A$.
\item 
$\uEqTens(\cC|A)$ is the categorical 1-group whose 1-morphisms are triples $(\alpha,\psi^\alpha, \lambda^\alpha)$ with $(\alpha,\psi^\alpha) \in \uEqTens(\cC)$ 
and $\lambda^\alpha: \alpha(A) \to A$ is an algebra isomorphism
with \emph{strictly associative} composition of 1-morphisms given by 
$$
(\alpha,\psi^\alpha, \lambda^\alpha)\circ (\gamma, \psi^\gamma, \lambda^\gamma) 
:= 
(\alpha \circ \gamma 
, 
\alpha(\psi^\gamma) \circ \psi^\alpha
, 
\lambda^\alpha \circ \alpha(\lambda^\gamma)),
$$ 
and whose 2-morphisms
$\eta: (\alpha,\psi^\alpha, \lambda^\alpha)\Rightarrow (\gamma,\psi^\gamma, \lambda^\gamma)$ are monoidal natural isomorphisms $\eta: (\alpha, \psi^\alpha) \Rightarrow (\gamma, \psi^\gamma)$ such that $\lambda^\gamma\circ \eta_A = \lambda^\alpha$.
\item
$\EqTens(\cC|A)$ is the truncated group, whose elements are equivalence classes of pairs $(\alpha, \lambda^\alpha)$ as above.
We denote the equivalence class of $(\alpha, \lambda^\alpha)$ by $[\alpha,\lambda^\alpha]$.
\end{itemize}
\end{defn}

Recall from \cite[Prop.~4.14.3]{MR3242743} that $\pi_{2}(\uEqTens(\cC)):=\Aut_\otimes(\id_\cC)\cong \Hom(\cG \to \bbC^\times)$, where $\cG$ is the universal grading group of the fusion category $\cC$.
Since $\uEqStab(A)\subseteq \uEqTens(\cC)$ is a full categorical 1-subgroup, it has the same $\pi_{2}$. 

We say a character $\gamma\in \Hom(\cG \to \bbC^\times)$ \emph{trivializes} the object $a\in \cC$ if $\gamma(g)=1$ for any homogenous $g$-graded subobject $b_{g}\subset a$. 
Under the isomorphism $\pi_{2}(\uEqTens(\cC))\cong \Hom(\cG \to \bbC^\times)$, it is easy to verify that 
$$
\pi_{2}(\uEqTens(\cC|A))
\cong
\set{\gamma\in  \Hom(\cG \to \bbC^\times)}{\gamma \text{ trivializes } A}.
$$
The following result is analogous to \cite[Cor.~3.7]{MR3354332} in the presence of an algebra object $A\in \cC$.

\begin{prop}
\label{prop:DecategorifiedExactSequenceNonBraided}
Let
$$
\begin{tikzcd}
\uEqTens(\cC|A)
\ar[]{r}{\underline{F}}
&
\uEqStab(A)
\end{tikzcd}
$$
be the forgetful monoidal functor $(\alpha, \lambda^\alpha)\mapsto \alpha$. 
There is a decategorified exact sequence of groups
$$
\begin{tikzcd}[column sep=small]
1
\ar[]{r}
&
\pi_{2}(\uEqTens(\cC|A))
\ar[hook]{r}{}
&
\Hom(\cG \to \bbC^\times)
\ar[]{r}{}
&
\Aut_\cC(A)
\ar[]{rr}{[\id_\cC, \,\cdot\,]}
&&
\EqTens(\cC|A)
\ar[]{r}{F}
&
\EqStab(A)
\ar[]{r}{}
&
1.
\end{tikzcd}
$$
\end{prop}
\begin{proof}
 The first injection is described right before the proposition. 
Given a character of the universal grading group $\gamma\in \Hom(\cG \to \bbC^\times)$, we define an automorphism of the object $A\in \cC$ by $\lambda_\gamma := \sum_{a\prec A}\gamma(\operatorname{gr}(a))\id_{a}$, where $\operatorname{gr}(a)\in \cG$ is the group element grading the simple object $a\in \Irr(\cC)$ appearing in the direct sum decomposition of $c$. 
Since $\operatorname{gr}(a)\operatorname{gr}(b) = \operatorname{gr}(c)$ for any subobject $c\prec a\otimes b$, $\lambda_\gamma$ is an algebra automorphism. 
Furthermore, characters which act trivially are obviously identified with $\pi_{2}(\uEqTens(\cC|A))$ as above. 
 
The map $[\id_\cC, \,\cdot\,]: \Aut_{\cC}(A)\rightarrow \EqTens(\cC|A)$ sends $\psi\in \Aut_{C}(A)$ to $[\id_{\cC},\psi]$. This is trivial precisely if we have a monoidal natural automorphism of the identity functor $\eta:\id_{\cC}\rightarrow \id_{\cC}$ such that $\eta_{A}=\psi$. 
But recall there is a canonical isomorphism between $\Aut_\otimes(\id_\cC)\cong \Hom(\cG \to \bbC^\times)$ by \cite[Prop. 4.14.3]{MR3242743}.
Thus the kernel of $[\id_\cC, \,\cdot\,]$ is precisely the set of automorphisms that are restrictions of the characters of the universal grading group. 

Finally, the kernel of the forgetful functor $F$ is exactly the image of $\Aut_{\cC}(A)$, and $F$ is surjective by definition of $\EqStab(A)$.
\end{proof}

\begin{defn}
\label{defn:SecondLevelCanonicalSurjection}
We now define a functor of categorical 1-groups
$\underline{F}_A:\uEqTens(\cC|A) \to \uEq(\cC_A)$.
First, given $(\alpha,\psi^\alpha, \lambda^\alpha)\in \uEqTens(\cC|A)$, we define $\underline{F}_A(\alpha,\psi^\alpha, \lambda) \in \uEq(\cC_A)$ as follows.
Given a right $A$-module $(M, r: M \otimes A \to A)$, the object $\alpha(M)\in \cC$ carries the structure of a right $A$-module via 
\begin{equation}
\label{eq:TwistRightAModuleByAut}
\alpha(r)\circ \psi^\alpha_{M, A}\circ (\id_{\alpha(M)}\otimes (\lambda^\alpha)^{-1}).
\end{equation}
It is straightforward to verify that for $f\in \cC_A(M\to N)$, $\alpha(f)\in \cC_A(\alpha(M) \to \alpha(N))$.
Hence $\underline{F}_A(\alpha, \psi^\alpha, \lambda^\alpha)$ is a well-defined object in $\uEq(\cC_A)$.
Notice that if $\eta: (\alpha, \psi^\alpha, \lambda^\alpha) \Rightarrow (\gamma, \psi^\gamma, \lambda^\gamma)$ and $(M,r)\in \cC_A$, then $\eta_M : \alpha(M) \to \gamma(M)$ is an $A$-module map.
Hence we may define 
$\underline{F}_A(\eta)_{(M,r)}:= \underline{F}_A(\eta_M)$.
Finally, it straightforward to show that $\underline{F}_A$ is a strict monoidal functor under the above definitions.
\end{defn}

\begin{defn}
\label{defn:FAInPresenceOfHalfBraiding}
Now suppose the connected separable algebra $A$ lifts to a commutative algebra $(A, e_A)\in \cZ(\cC)$.
Given $(\alpha ,\psi^\alpha)\in \uEqTens(\cC)$ and $(A,e_A)\in \cZ(\cC)$, there is a canonical lift of $\alpha(A)$ to $\cZ(\cC)$ given by
\begin{equation}
\label{eq:HalfBraidingForAlphaA}
e_{\alpha(A),\alpha(c)} 
:=
(\psi^{\alpha}_{c,A})^{-1}\circ \alpha(e_{A,c})\circ \psi^\alpha_{A,c}.
\end{equation}
Moreover, the multiplication $m_{\alpha(A)}:= \alpha(m_A)\circ\psi^\alpha_{A,A}$ and unit $i_{\alpha(A)}:= \alpha(i_A)$ are automatically morphisms in $\cZ(\cC)$ using the half-braiding $e_{\alpha(A)}$ from \eqref{eq:HalfBraidingForAlphaA}.

We define:
\begin{itemize}
\item
$\uEqStab(A,e_A)\subset \uEqStab(A)$ is the full categorical subgroup whose 1-morphisms $\alpha\in \uEqTens(\cC)$ satisfy that there exists an algebra isomorphism $\lambda: \cZ(\cC)(\alpha(A)\to A)$.
Notice that $\lambda$ being a morphism in $Z(\cC)$ is equivalent to
\begin{equation}
\label{eq:HalfBraidingCompatibility}
e_{\alpha(A),\alpha(c)}
=
(\id_{\alpha(c)} \otimes \lambda^{-1})\circ e_{A,c}\circ (\lambda \otimes \id_{\alpha(c)})
\qquad\qquad
\forall c\in \cC.
\end{equation}

\item 
$\uEqTens(\cC|A,e_A)\subset \uEqTens(\cC|A)$ is the full categorical subgroup whose 1-morphisms $(\alpha,\psi^\alpha, \lambda^\alpha)$ are such that $\lambda^\alpha \in \cZ(\cC)(\alpha(A) \to A)$ is an algebra isomorphism.
\end{itemize}

We now define the underlying functor $\underline{F}_A^\cZ(\alpha, \psi^\alpha, \lambda^\alpha)\in \uEq(\cC_A)$ as in Definition \ref{defn:SecondLevelCanonicalSurjection}.
\end{defn}

\begin{lem}
\label{lem:TensoratorCompatibilityWithSeparabilityIdempotent}
The tensorator $\psi^\alpha$ is compatible with the separability idempotents \eqref{eq:SeparabilityIdempotent}
$$
\psi^\alpha_{M,N} \circ p_{\alpha(M), \alpha(N)} 
= 
\alpha(p_{M,N}) \circ \psi^\alpha_{M,N}
\qquad\qquad
\forall M,N\in \cC_A.
$$
Thus $\psi^\alpha$ endows $\underline{F}_A^\cZ(\alpha, \psi^\alpha, \lambda^\alpha)$ with the structure of a tensor functor in $\uEqTens(\cC_A)$.
\end{lem}
\begin{proof}
We write $\psi=\psi^\alpha$ and $\lambda=\lambda^{\alpha}$ to ease the notation.
We also write a $\psi$ with multiple subscripts to denote a composite of $\psi$'s, as the order does not matter via associativity of $\psi$.
Let $s: A\to A\otimes A$ be the splitting of $m: A\otimes A \to A$ which is unique as $A$ is connected.
By the uniqueness of $s$ and as $\lambda: \alpha(A) \to A$ is an algebra isomorphism, we see that
\begin{equation}
\label{eq:SplittingCompatibilityWithLambda}
(\lambda^{-1}\otimes \lambda^{-1}) \circ s \circ i
= 
\psi_{A,A}^{-1}\circ \alpha(s)\circ \lambda^{-1}\circ i
=
\psi_{A,A}^{-1}\circ\alpha(s)\circ \alpha(i).
\end{equation}
We calculate
\begin{align*}
\psi_{M,N} &\circ p_{\alpha(M), \alpha(N)} 
\\&=
\psi_{M,N}
\circ
(r_{\alpha(M)} \otimes r_{\alpha(N)})
\circ
(\id_{\alpha(M)\otimes A} \otimes e_{A,\alpha(N)})
\circ
(\id_{\alpha(M)} \otimes (s\circ i) \otimes \id_{\alpha(N)})
\\&=
\alpha(r_M\otimes r_N)
\circ
\psi_{M,A,N,A}
\\&\hspace{1cm}\circ
(\id_{\alpha(M)}\otimes \lambda^{-1}\otimes ((\id_{\alpha(N)}\otimes \lambda^{-1})\circ e_{A,\alpha(N)})\otimes \id_{A})
\circ
(\id_{\alpha(M)} \otimes (s\circ i) \otimes \id_{\alpha(N)})
\\&\underset{\eqref{eq:HalfBraidingCompatibility}}{=}
\alpha(r_M\otimes r_N)
\circ
\psi_{M,A,N,A}
\\&\hspace{1cm}\circ
(\id_{\alpha(M)}\otimes \lambda^{-1}\otimes (e_{\alpha(A),\alpha(N)}\circ (\lambda^{-1}\otimes \id_{\alpha(N)}))\otimes \id_{A})
\circ
(\id_{\alpha(M)} \otimes (s\circ i) \otimes \id_{\alpha(N)})
\\&\underset{\eqref{eq:HalfBraidingForAlphaA}}{=}
\alpha(r_M\otimes r_N)
\circ
\psi_{M,A,N,A}
\circ
(\id_{\alpha(M)}\otimes \id_A\otimes (\psi^{-1}_{N,A}\circ \alpha(e_{A,N})\circ \psi_{A,N}))
\\&\hspace{1cm}
\circ
(\id_{\alpha(M)} \otimes ((\lambda^{-1}\otimes \lambda^{-1})\circ s\circ i) \otimes \id_{\alpha(N)})
\\&\underset{\eqref{eq:SplittingCompatibilityWithLambda}}{=}
\alpha(r_M\otimes r_N)
\circ
\psi_{M,A,N,A}
\circ
(\id_{\alpha(M)}\otimes \id_A\otimes (\psi^{-1}_{N,A}\circ \alpha(e_{A,N})\circ \psi_{A,N}))
\\&\hspace{1cm}
\circ
(\id_{\alpha(M)}\otimes(\psi^{-1}_{A,A}\circ \alpha(s)\circ\alpha(i))\otimes \id_{\alpha(N)})
\\&=
\alpha(p_{M,N})\circ \psi_{M,N}.
\qedhere
\end{align*}
\end{proof}

We conclude that $\underline{F}_A^\cZ$ is a strict monoidal functor of categorical groups.

We now show that a central $G$-equivariant structure $\lambda$ on $A$ exactly correspond to liftings of the categorical action $\uG\to \uEqStab(A,e_A)$ to $\uEqTens(\cC|A,e_A)$.

\begin{equation}
\label{eq:LiftGAction}
\begin{tikzcd}
\uG
\ar[]{dr}{{(\underline{\rho},\mu)}}
\ar[dd,dashed, "\exists?"]
\\
&
\uEqStab(A,e_A)
\ar[hook]{r}{}
& 
\uEqStab(A)
\ar[hook]{r}{}
&
\uEqTens(\cC)
\\
\uEqTens(\cC|A,e_A)
\ar[hook]{r}{}
\ar[two heads]{ur}{\underline{F}^Z}
\ar[swap]{dr}{\underline{F}_A^Z}
&
\uEqTens(\cC|A)
\ar[swap, two heads]{ur}{\underline{F}}
\ar[dr, "\underline{F}_A"]
\\
&
\uEqTens(\cC_A)
\ar[swap, two heads]{r}{}
& 
\uEq(\cC_A)
\end{tikzcd}
\end{equation}


\begin{prop}
\label{prop:LiftingCategoricalGAction}
The data of a lifting of $(\underline{\rho},\mu)$ to $\uEqTens(\cC|A,e_A)$ is equivalent to a central $G$-equivariant algebra structure $\lambda$ on $(A,e)$.
\end{prop}
\begin{proof}
A lifting of $(\underline{\rho},\mu)$ to a monoidal functor $(\underline{\widetilde{\rho}},\mu): \uG \to \uEqTens(\cC|A,e_A)$
is exactly the data of an algebra isomorphism $\lambda^g \in \cZ(\cC)(g(A) \to A)$ for all $g\in G$ which satisfies the monoidality axiom \eqref{eq:EquivariantObject}.
\end{proof}

\begin{rem}
\label{rem:BoringLift}
Given a categorical action $\uG\to \uEqStab(A)$ in the absence of a half braiding on $A$, the data of a lift $\uG \to \uEqTens(\cC|A)$ is equivalent to a $G$-equivariant algebra structure $\lambda$ on $A$.
Since we will not use this result, we leave it to the interested reader.
\end{rem}

We conclude that given a $G$-equivariant structure $\lambda$ on $A$, 
$(\underline{\rho},\mu):\underline{G}\rightarrow \uEqTens(\mathcal{C})$
induces a canonical categorical action
$(\underline{\rho}_A,\mu):\underline{G}\rightarrow \uEq(\mathcal{C}_A)$ given by postcomposing the lift of $(\underline{\rho}, \mu)$ from Remark \ref{rem:BoringLift} with $\underline{F}_A$.
Moreover, in the presence of a half-braiding $e_A$ on $A$ such that $(A,e_A)$ is commutative, given a central $G$-equivariant structure $\lambda$ on $A$, Proposition \ref{prop:LiftingCategoricalGAction} gives us a canonical categorical action $(\underline{\rho}_A^\cZ, \mu): \uG \to \uEqTens(\cC_A)$.

\subsection{The induced action when \texorpdfstring{$\cC$}{C} is non-degenerately braided}
\label{sec:InducedActionForEtaleAlgebra}

Suppose now that $\cC$ is a non-degenerately braided fusion category and $(A,m,i)\in \cC$ is a connected \'{e}tale algebra.
In this setting, we define
\begin{itemize}
\item 
$\uEqBrStab(A)\subset \uEqBr(\cC)$ is the full categorical subgroup whose 1-morphisms are braided autoequivalences $(\alpha, \psi^\alpha) \in \uEqBr(\cC)$ such that there exists an algebra isomorphism $\alpha(A) \cong A$, and whose 2-morphisms are monoidal natural isomorphisms.
\item
$\uEqBr(\cC|A)$ is the categorical group whose 1-morphisms are tuples $(\alpha,\psi^\alpha, \lambda^\alpha)$ such that $(\alpha, \psi^\alpha)\in\uEqBr(\cC)$ and $\lambda^\alpha: \alpha(A) \to A$ is an algebra isomorphism, and whose 2-morphisms $\eta: (\alpha,\psi^\alpha, \lambda^\alpha)\Rightarrow (\gamma,\psi^\gamma, \lambda^\gamma)$ are monoidal natural isomorphisms $\eta: (\alpha, \psi^\alpha) \Rightarrow (\gamma, \psi^\gamma)$ such that $\lambda^\gamma\circ \eta_A = \lambda^\alpha$.
\end{itemize}

\begin{rem}
\label{rem:Pi2OfEqBr(C)}
Recall from \cite[{Prop.~7.3.ii}]{MR2677836} that we have a group isomorphism $\sigma:\Inv(\cC)\to \pi_2(\uEqBr(\cC))$ denoted $a\mapsto \sigma^a$
which is determined by its value on a simple object $c\in \cC$ by the morphism in $\End_\cC(c)$ such that $\beta_{c,a}\circ \beta_{a,c} = \id_a \otimes\sigma^a_c$.
The isomorphism $\sigma^a$ is called the \emph{loop/ring operator} because when $\cC$ has a pivotal structure, we may represent $\sigma^a_c$ for arbitrary $c\in \cC$ graphically by \begin{equation*}
\sigma^a_c=
\begin{tikzpicture}[baseline=-.1cm]
	\draw (0,-.1) -- (0,.4);
	\draw (115:.2cm) arc (115:425:.2cm);
	\draw (0,-.3) -- (0,-.5);
	\node at (0,.6) {\scriptsize{$c$}};
	\node at (-.35,0) {\scriptsize{$a$}};
\end{tikzpicture}\,.
\end{equation*}
Notice that since $\uEqBrStab(A)\subset \uEqBr(\cC)$ is a full categorical subgroup, $\pi_2(\uEqBrStab(A)) = \pi_2(\uEqBr(\cC))\cong \Inv(\cC)$. 

Now given a fixed $c\in \cC$, we have an exact sequence
\begin{equation}
\label{eq:RingOperatorSeqeunce}
\begin{tikzcd}
1
\ar[]{r}{}
&
c^{\prime}\cap \operatorname{Inv}(\cC)
\ar[hook]{r}
&
\operatorname{Inv}(\cC)
\ar[]{r}{\sigma_c}
&
\Aut_\cC(c)
\end{tikzcd}
\end{equation}
where $\sigma_c(a):= \sigma^a_c$, and $c'\cap \Inv(\cC)$ is the subgroup of $a\in \Inv(\cC)$ which are \emph{transparent} to $c$, i.e., $\beta_{c,a}\circ \beta_{a,c} = \id_{a\otimes c}$.
\end{rem}

\begin{lem}
The group isomorphism $\sigma:\Inv(\cC) \to \pi_2(\uEqBr(\cC))$ from Remark \ref{rem:Pi2OfEqBr(C)} yields a group isomorphism $\pi_2(\uEqBr(\cC|A)) \cong A'\cap \Inv(\cC)$.
\end{lem}
\begin{proof}
Suppose $\eta$ is an automorphism of the identity $(\id_\cC,\id_A)\in \uEqBr(\cC|A)$.
Then by Remark \ref{rem:Pi2OfEqBr(C)}, $\eta \in \pi_2(\uEqBr(\cC))$, so $\eta= \sigma^a$ for some $a\in \Inv(\cC)$.
Now the condition that $\id_A \circ \eta_A = \id_A$ implies that $\sigma^a_A = \id_A$, and thus $a$ must be transparent to $A$.
\end{proof}

The following proposition is analogous to Proposition \ref{prop:DecategorifiedExactSequenceNonBraided}.

\begin{prop}
Let
$$
\begin{tikzcd}
\uEqBr(\cC|A)
\ar[]{r}{\underline{F}}
&
\uEqBrStab(A)
\end{tikzcd}
$$
be the forgetful functor. 
Then \eqref{eq:RingOperatorSeqeunce} extends to an exact sequence of groups
$$
\begin{tikzcd}[column sep=small]
1
\ar[]{r}{}
&
A^{\prime}\cap \operatorname{Inv}(\cC)
\ar[hook]{r}
&
\operatorname{Inv}(\cC)
\ar[]{r}{\sigma_A}
&
\Aut_\cC(A)
\ar[]{rr}{[\id_\cC, \,\cdot\,]}
&&
\EqBr(\cC|A)
\ar[]{r}{F}
&
\EqBrStab_\cC(A)
\ar[]{r}{}
&
1.
\end{tikzcd}
$$
\end{prop}
\begin{proof}
Observe $\lambda$ is in the kernel of the map $\lambda \mapsto [\id_{\cC}, \lambda]\in \EqBr(\cC|A)$ if and only if 
there is a monoidal natural automorphism $\eta:\id_\cC\Rightarrow \id_\cC$ such that $\lambda\circ \eta_{A}=\id_{A}$. 
But by Remark \ref{rem:Pi2OfEqBr(C)}, 
when $\cC$ is non-degenerately braided, 
the map $\Inv(\cC)\ni a \mapsto \sigma^a \in \Aut_\otimes^{\rm br}(\id_\cC)$ is an isomorphism.
Thus the kernel of $[\id_\cC,\,\cdot\,]$ is precisely the image of $\sigma_A$. 
Finally, notice that kernel of $F$ is the set of $[\id_{\cC}, \lambda]$ such that $\lambda:A\rightarrow A$ is an algebra isomorphism, which is precisely the image of $[\id_\cC, \,\cdot\,]: \Aut_\cC(A) \to \EqTens(\cC|A)$.
\end{proof}

Analogous to Definition \ref{defn:FAInPresenceOfHalfBraiding} and Lemma \ref{lem:TensoratorCompatibilityWithSeparabilityIdempotent}, there is a strict monoidal functor
$$
\underline{F}_A: \uEqBr(\cC|A) \to \uEqTens(\cC_A).
$$
Since every algebra isomorphism $\lambda^\alpha\in \cC(\alpha(A) \to A)$ is compatible with the braiding, we no longer need to pass to $Z(\cC)$.

\begin{lem}
Suppose $(\alpha, \psi^\alpha, \lambda^\alpha)\in \uEqBr(\cC|A)$.
The functor $\underline{F}_A(\alpha, \psi^\alpha, \lambda^\alpha)\in \uEqTens(\cC_A)$ preserves local modules in $\cC_A^\loc$ and is compatible with the braiding on $\cC_A^\loc$.
Thus $\underline{F}_A(\alpha, \psi^\alpha, \lambda^\alpha)\in\uEqBr(\cC_A^\loc)$.
\end{lem}
\begin{proof}
We write $\psi = \psi^\alpha$ and $\lambda = \lambda^{\alpha}$ to ease the notation.
Suppose $(M,r)\in \cC_A^\loc$.
Then since $(\alpha,\psi)\in \uEqBr(\cC)$, we have
\begin{align*}
\alpha(r)
\circ 
\psi_{M,A}
&\circ 
(\id_{\alpha(M)}\otimes \lambda^{-1})
\circ
\beta_{A,\alpha(M)}
\circ
\beta_{\alpha(M), A}
\\&=
\alpha(r)
\circ 
\psi_{M,A}
\circ
\beta_{\alpha(A),\alpha(M)}
\circ
\beta_{\alpha(M), \alpha(A)}
\circ 
(\id_{\alpha(M)}\otimes \lambda^{-1})
\\&=
\alpha(
r
\circ 
\beta_{A,M}\circ \beta_{M,A}
)
\circ
\psi_{M,A}
\circ 
(\id_{\alpha(M)}\otimes \lambda^{-1})
\\&=
\alpha(r)
\circ
\psi_{M,A}
\circ 
(\id_{\alpha(M)}\otimes \lambda^{-1}),
\end{align*}
and thus $\alpha(M)\in \cC_A^\loc$.
As the tensorator of $\underline{F}_A(\alpha, \psi, \lambda)$ is exactly $\psi$, $(\alpha,\psi)\in \uEqBr(\cC)$, and the braiding of $\cC_A^\loc$ is exactly the braiding on $\cC$, we have that $\underline{F}_A(\alpha, \psi, \lambda)$ is braided.
\end{proof}

We now see that analogous to Proposition \ref{prop:LiftingCategoricalGAction}, $G$-equivariant structures $\lambda$ on $A$ correspond to lifts of $(\underline{\rho},\mu): \uG \to \uEqBrStab(A)$
to 
$(\widetilde{\underline{\rho}}, \mu):\uG \to \uEqBr(\cC|A)$.

\begin{equation}
\label{eq:LiftGActionBraided}
\begin{tikzcd}
\uG
\ar[]{dr}{{(\underline{\rho},\mu)}}
\ar[dd,dashed, "\exists?"]
\\
& 
\uEqBrStab(A)
\ar[hook]{r}{}
&
\uEqBr(\cC)
\\
\uEqBr(\cC|A)
\ar[swap, two heads]{ur}{\underline{F}}
\ar[dr, "\underline{F}_A"]
\\
&
\uEqBr(\cC_A^\loc)
\end{tikzcd}
\end{equation}

\subsection{Induced actions on \texorpdfstring{$G$}{G}-crossed braided extensions}
\label{sec:InducedActionForGCrossed}

For this section, suppose $\cC$ is a non-degenerately braided fusion category, 
$A\in \cC$ is a connected \'{e}tale algebra, and
$\cF = \bigoplus_{g\in G} \cF_g$
is a $G$-crossed braided extension of $\cC=\cF_e$.
We also assume we have a fixed $G$-equivariant algebra structure $\lambda$ on $A$.
By \cite{MR3671186} (see Notation \ref{nota:CategoricalAction}) and Remark \ref{rem:StrictlyUnitalEquivariantStructure}, $\lambda^e = \id_A$.
We claim we can construct a canonical $G$-crossed braided extension $\cE$ of $\cC_A^\loc$, which satisfies that the equivariantization $\cE^G$ can be canonically identified as a braided tensor category with $(\cF^G)_{(A,\lambda)}^{\loc}$.
That is, the following diagram commutes:
\begin{equation}
\label{eq:GaugingCommutes}
\begin{tikzcd}
  & \cC_A^{\mathrm{loc}} \arrow[rd, "\text{gauge lifted }G\text{ action}"] &  
  \\
  \cC \arrow[rd, "\text{gauge }G\text{ action}"'] \arrow[ru,squiggly, "\text{condense }A"] &  & \cE^G 
  \\
  & \cF^G \arrow[ru,squiggly,"\text{condense }{(A,\lambda)}"']
  &
\end{tikzcd}
\end{equation}

First, notice that since $A\in \cC= \cF_e$, the $G$-crossed braiding $\beta$ of $\cF$ provides $A$ with a lift to $Z(\cF)$, i.e., $(A, \beta_{A,-}) \in Z(\cF)$.
We immediately see that $\cF_A$ is a tensor category.

\begin{rem}
Observe that since $A\in \cC = \cF_e$, any right $A$-module decomposes as a canonical direct sum of $g$-graded right $A$-modules, and that $(\cF_g)_A$ is a semisimple category for all $g\in G$.
Hence $\cF_A = \bigoplus_{g\in G} (\cF_g)_A$ is manifestly a $G$-extension of the tensor category $\cC_A = (\cF_e)_A$.
\end{rem}

As $\cF$ is a $G$-crossed braided extension of $\cC$, we have a categorical action $\uG \to \uEqTens(\cF)$.
Notice that our $G$-equivariant algbra structure $\lambda$ on $A\in \cC$ is \emph{central} in $\cF$, as $\cC=\cF_e$. 
By Proposition \ref{prop:LiftingCategoricalGAction}, $\lambda$ gives us a lift of our categorical action to $\uG \to \uEqTens(\cF|A)$.
Postcomposing with the strict monoidal functor $\underline{F}_A: \uEqTens(\cF|A) \to \uEqTens(\cF_A)$, we get a categorical action $(\underline{\rho}_A, \mu):\uG \to \uEqTens(\cF_A)$.
We now find a subcategory of $\cF_A$ which is a $G$-crossed braided extension of $\cC_A^{\loc}$.

\begin{defn}
\label{defn:gLocalModules}
For $g\in G$, we define the category of $g$-\emph{local modules} $\cF_A^{g-\loc}\subset (\cF_g)_A$ to be the right $A$-modules $(M,r)\in (\cF_g)_A$ such that
\begin{equation}
\label{eq:gLocalModule}
r=
r\circ \beta_{A,M} \circ (\lambda^g \otimes \id_M)\circ \beta_{M,A}.
\end{equation}
Notice $\cF_A^{g-\loc} \subset (\cF_g)_A$ is a semisimple subcategory. 

We define $\cE := \bigoplus_{g\in G} \cF_A^{g-\loc}$.
It is straightforward to verify the the full subcategory $\cE\subseteq \cF_{A}$ is closed under $\otimes_{A}$. 
Hence $\cE$ is manifestly a $G$-extension of $\cE_e = \cF_A^{e-\loc} = \cC^{\loc}_{A}$.
One now verifies that the $g$-action on $\cF_A$ maps $h$-local modules to $ghg^{-1}$-local modules.
Hence $\cE$ carries a categorical $G$-action.
\end{defn}

\begin{lem}
For all $(M,r_M)\in \cF_A^{g-\loc}$ and $(N,r_N)\in \cF_A^{h-\loc}$, $\beta_{M,N}\circ p_{M,N}=p_{g(N),M} \circ \beta_{M,N}$ where $p$ denotes the corresponding separability idempotents from \eqref{eq:SeparabilityIdempotent}.
Hence the $G$-crossed braiding $\beta$ of $\cF$ descends to a $G$-crossed braiding of $\cE$.
\end{lem}
\begin{proof}
Below, we write $\psi = \psi^{g}$ and $\lambda = \lambda^g$ to ease the notation.
Let $s\in \Hom_{A-A}(A\to A\otimes A)$ be the splitting of $m\in\Hom_{A-A}(A\otimes A \to A)$.
Since $s$ is unique as $A$ is connected, \eqref{eq:SplittingCompatibilityWithLambda} holds.
Moreover, by uniqueness of $s$ and commutativity of $A$ ($m\circ \beta_{A,A}=m$), we have
\begin{equation}
\label{eq:SplittingCommutative}
\beta_{A,A}\circ s = s.
\end{equation}
We now calculate that
\begin{align*}
&\beta_{M,N}\circ p_{M,N}
\\&=
\beta_{M,N}
\circ
(r_M \otimes r_N)
\circ
(\id_{M\otimes A} \otimes \beta_{A,N})
\circ
(\id_M \otimes (s\circ i) \otimes \id_N)
\\&=
(\id_{g(N)}\otimes r_M)
\circ
(\beta_{M,N}\otimes \id_{A})
\circ
(\id_M\otimes \beta_{A,N})
\circ
(\id_{M\otimes A}\otimes (r_N\circ \beta_{A,N}))
\circ
(\id_M \otimes (s\circ i) \otimes \id_N)
\\&=
(\id_{g(N)}\otimes r_M)
\circ
(\beta_{M,N}\otimes \id_{A})
\circ
(\id_M\otimes r_N \otimes \id_A)
\circ
(\id_{M\otimes N}\otimes (\beta_{A,A}\circ s\circ i))
\\&
\underset{\eqref{eq:SplittingCommutative}}{=}
(\id_{g(N)}\otimes r_M)
\circ
(\beta_{M,N}\otimes \id_{A})
\circ
(\id_M\otimes r_N \otimes \id_A)
\circ
(\id_{M\otimes N}\otimes (s\circ i))
\\&=
((g(r_N)\circ \psi_{N,A})\otimes (r_M\circ \beta_{A,M}^{-1} \circ (\lambda^{-1} \otimes \id_M)))
\circ
(\id_{M}\otimes ((\id_A\otimes \lambda)\circ \psi_{A,A}^{-1}\circ g(s\circ i))\otimes_N)
\circ
\beta_{M,N}
\\&\underset{\eqref{eq:gLocalModule}}{=}
((g(r_N)\circ \psi_{N,A})\otimes (r_M\circ \beta_{A,M})
\circ
(\id_{M}\otimes ((\id_A\otimes \lambda)\circ \psi_{A,A}^{-1}\circ g(s\circ i))\otimes_N)
\circ
\beta_{M,N}
\\&\underset{\eqref{eq:SplittingCompatibilityWithLambda}}{=}
((g(r_N)\circ \psi_{N,A})\otimes (r_M\circ \beta_{A,M})
\circ
(\id_{M}\otimes ((\lambda^{-1}\otimes\id_A)\circ s\circ i)\otimes_N)
\circ
\beta_{M,N}
\\&=
p_{g(N),M}
\circ
\beta_{M,N}.
\end{align*}
Hence $\cE$ is a $G$-crossed braided extension of $\cE_e=\cF^{e-\loc}_A=\cC_A^\loc$.
\end{proof}

Since our $G$-equivariant structure $\lambda$ on $A$ gave us a categorical action $(\underline{\rho}_A, \mu): \uG \to \uEqTens(\cF_A)$, we can take the $G$-equivariantization, which we denote by $(\cF_A)^{G,\lambda}$ to emphasize
that the categorical $G$-action on $\cF_A$ depends on $\lambda$.

\begin{thm}\label{thm:GaugingCommutes}
The equivariantization $(\cF_A)^{G,\lambda}$ is tensor \emph{isomorphic} to the tensor category $\cF^G_{(A,\lambda)}$ viewed as subcategories of $\cF$.
Moreover, this isomorphism restricts to a braided tensor isomorphism $\cE^{G,\lambda}\cong (\cF^G)_{(A,\lambda)}^\loc$. 
\end{thm}
\begin{proof}
One one side, the objects of $(\cF_A)^{G,\lambda}$ are tuples $((M,r_M),\gamma)$ where $(M,r_M)\in \cF_A$ is a right $A$-module and $\gamma=\{\gamma_g \in \cF_A(g(M)\to M)\}_{g\in G}$ is a family of $G$-equivariant $A$-module isomorphisms.
Using the definition of the $A$-action on $g(M)$ from \eqref{eq:TwistRightAModuleByAut}, this means that the $\gamma_g: g(M)\to M$ satisfy
\begin{equation}
\label{eq:AModObjectWhichIsGEquivariant}
\gamma_g
\circ
g(r_M)\circ \psi^g_{M,A}\circ (\id_{g(M)}\otimes (\lambda^g)^{-1})
=
r_M\circ\gamma_g.
\end{equation}

The morphisms of $(\cF_A)^{G,\lambda}$ are $A$-module maps which are $G$-equivariant.
This means the morphisms $f\in \cF_A(((M,r_M),\gamma^M) \to ((N,r_N),\gamma^N))$ must satisfy the two conditions
\begin{equation}
\label{eq:GEquivariantAModuleMorphisms}
f\circ r_M = r_N \circ f
\qquad\qquad
\text{and}
\qquad\qquad
\gamma^N_g\circ g(f) \underset{\eqref{eq:EquivariantMorphism}}{=} f\circ \gamma^M_g.
\end{equation}

The tensor structure on $(\cF_A)^{G,\lambda}$ is given by 
\begin{equation}
\label{eq:TensorProductOnGEquivariantAMods}
\begin{split}
((M,r_M),\gamma^M)&\otimes_{(\cF_A)^{G,\lambda}} ((N,r_N),\gamma^N)
\\&=
((M,r_M)\otimes_{(\cF_A)^{G,\lambda}}(N,r_N),p_{M,N}\circ (\gamma^M\otimes \gamma^N)\circ (\psi^g_{M,N})^{-1})
\\&=
((M\otimes_A N, p_{M,N}\circ(r_M\otimes r_N)),p_{M,N}\circ (\gamma^M\otimes \gamma^N)\circ (\psi^g_{M,N})^{-1})
\end{split}
\end{equation}
where $p_{M,N}$ is the separability idempotent \eqref{eq:SeparabilityIdempotent}.

On the other side, the objects of $(\cF^G)_{(A,\lambda)}$ are tuples $((M,\gamma),r_M)$ where $(M,\gamma)\in \cF^G$, so $M\in \cF$ and $\gamma=\{\gamma_g \in \cF(g(M)\to M)\}_{g\in G}$ is a family of $G$-equivariant isomorphisms, and $r_M \in \cF^G((M,\gamma) \otimes_{\cF^G} (A,\lambda)\to (M,\gamma))$ is a $G$-equivariant map.
Thus by 
\eqref{eq:EquivariantMorphism}
and
\eqref{eq:TensorProductOnEquivariantization},
\begin{equation*}
r_M\circ (\gamma_g\otimes \lambda^g)\circ (\psi^g_{M,A})^{-1}
=
\gamma_g\circ g(r_M),
\end{equation*}
which is equivalent to \eqref{eq:GEquivariantAModuleMorphisms}.
Now $r_M$ is an $(A,\lambda)$-module map, which is exactly the requirement that $(M,r_M)\in \cF_A$.
Hence the object $((M,\gamma),r_M)\in (\cF^G)_{(A,\lambda)}$ can be canonically identified with $((M,r_M),\gamma)\in (\cF_A)^{G,\lambda}$.

Now the morphisms of $(\cF^G)_{(A,\lambda)}$ are $G$-equivariant maps which are compatible with the $(A,\lambda)$-module structures, which is exactly the two conditions in \eqref{eq:GEquivariantAModuleMorphisms}.
Thus we may canonically identify the morphisms of $(\cF^G)_{(A,\lambda)}$ with the morphisms of $(\cF_A)^{G,\lambda}$, and this identification preserves composition.

The tensor structure on $(\cF^G)_{(A,\lambda)}$ is given by
\begin{align*}
((M,\gamma^M),r_M)&\otimes_{(\cF^G)_{(A,\lambda)}}((N,\gamma^N),r_N)
\\&=
((M,\gamma^M)\otimes_{(A,\lambda)}(N,\gamma^N) , p_{M,N}\circ (r_M \otimes r_N))
\\&=
((M\otimes_A N,
p_{M,N}\circ (\gamma^M\otimes_{\cF^G}\gamma^N)) , p_{M,N}\circ (r_M \otimes r_N))
\\&=
((M\otimes_A N,
p_{M,N}\circ (\gamma^M\otimes\gamma^N)\circ (\psi^g_{M,N})^{-1}) , p_{M,N}\circ (r_M \otimes r_N)).
\end{align*}
Thus we see from \eqref{eq:TensorProductOnGEquivariantAMods} that our identification of objects from $(\cF^G)_{(A,\lambda)}$ and $(\cF_A)^{G,\lambda}$ identifies the tensor products on the nose.
Hence the tensor categories $(\cF^G)_{(A,\lambda)}$ and $(\cF_A)^{G,\lambda}$ are canonically tensor isomorphic.

It remains to show that our tensor isomorphism descends to a braided tensor isomorphism $\cE^{G,\lambda}\cong (\cF^G)_{(A,\lambda)}^\loc$.
Suppose $(M,r,\gamma)\in (\cF^G)_{(A,\lambda)}\cong (\cF_A)^{G,\lambda}$.
Since $\cF$ is $G$-graded and $A\in \cF_e$, we can decompose $(M,r)\in \cF_A$ as $\bigoplus_{g\in G} (M_g, r_{g})$ where $(M_g,r_g)\in (\cF_g)_A$.
Now the braiding in $\cF^G(M\otimes A \to A\otimes M)$ is given by
$$
\bigoplus_{g\in G} (\lambda^g \otimes \id_{M_g} )\circ \beta_{M_g, A} \in \bigoplus_{g\in G} \cF( M_g\otimes A \to A\otimes M_g),
$$
where $\beta$ is the $G$-crossed braiding on $\cF$.
Since $A\in \cF_e$, the braiding in $\cF^G(A\otimes M \to M\otimes A)$ is exactly $\beta_{A,M}$.
Hence $(M,r,\gamma)$ is local in $(\cF^G)_{(A,\lambda)}\cong (\cF_A)^{G,\lambda}$ if and only if for each $g\in G$,
$$
r_{g}\circ \beta_{A,M_g} \circ (\lambda^g \otimes \id_{M_g} )\circ \beta_{M_g, A}
=
r_g,
$$
which is exactly the condition that the component module $(M_g,r_g) \in \cE_g=\cF_A^{g-\loc}$ for each $g\in G$.
Now the condition that $(M,r,\gamma)$ is $G$-equivariant is exactly the condition that $(M,r,\gamma)\in \cE^{G,\lambda}$.
Hence our isomorphism restricts to a tensor isomorphism $\cE^{G,\lambda}\cong (\cF^G)_{(A,\lambda)}^\loc$.
Since the braidings on both sides are exactly the $G$-crossed braidings on components, the tensor isomorphism $\cE^{G,\lambda}\cong (\cF^G)_{(A,\lambda)}^\loc$ is braided.
\end{proof}

In summary, we have the following commutative diagram, where we write $\cF_A^{G-\loc}$ for $\cE$ to keep the notation consistent.
\begin{equation}
\label{eq:OodlesOfCategories}
\begin{tikzcd}
&
\cF^G
\arrow[from=dd, hook'] 
\arrow[tail, shift left]{dl}
\arrow[shift left]{rr}{-\otimes (A,\lambda)} 
&&
\cF_A^G
\arrow[dashed, tail, shift left]{ll}{} 
\arrow[from=dd, hook'] 
\arrow[tail, shift left]{dl}
&&
(\cF_A^{G-\loc})^G
\arrow[hook']{ll}{} 
\arrow[tail, shift left]{dl}
\\
\cF
\arrow[dashed, shift left]{ur}
\arrow[shift left,crossing over]{rr}[near end]{-\otimes A} 
&& 
\cF_A 
\arrow[dashed, tail, crossing over, shift left]{ll}{}
\arrow[dashed, shift left]{ur}
&& 
\cF_A^{G-\loc}
\arrow[hook', crossing over]{ll}{} 
\arrow[dashed, shift left]{ur}
\\
&
\cC^G
\arrow[shift left]{rr}[near end]{-\otimes (A,\lambda)} 
\arrow[tail, shift left]{dl}
&&
\cC_A^G
\arrow[dashed, tail, shift left]{ll}{} 
\arrow[tail, shift left]{dl}
&&
(\cC_A^{\loc})^G
\arrow[hook']{uu}{} 
\arrow[hook']{ll}{} 
\arrow[tail, shift left]{dl}
\\
\cC
\arrow[hook]{uu}{} 
\arrow[dashed, shift left]{ur}
\arrow[shift left]{rr}{-\otimes A} 
&&  
\cC_A
\arrow[hook', crossing over]{uu}{} 
\arrow[dashed, tail, shift left]{ll}{} 
\arrow[dashed, shift left]{ur}
&&
\cC_A^{\loc} 
\arrow[hook']{ll}{} 
\arrow[hook', crossing over]{uu}{} 
\arrow[dashed, shift left]{ur}
\end{tikzcd}
\end{equation}

In the diagram \eqref{eq:OodlesOfCategories}, 
\begin{itemize}
\item
a hook arrow $\hookrightarrow$ denotes a fully faithful functor,
\item
a tail on an arrow $\rightarrowtail$ denotes a forgetful functor, and
\item
a dashed arrow $\dashrightarrow$ means a functor which is not a tensor functor.
\end{itemize}

\section{Examples}
\label{sec:Examples}

We now provide many examples of symmetry breaking from anyon condensation.
We begin by reviewing the Landau theory of symmetry breaking phase transitions in \S\ref{sec:LandauTheory}, which is trivial from the categorical perspective. 
We then discuss the toric code in \S\ref{sec:ToricCode}, which is the simplest example of a topological order in a condensed matter system. 
Here we already observe non-trivial categorical symmetry breaking when condensing the $e$ anyon.
We then continue with $G$-stable actions in \S\ref{sec:StableActions}, where $g =\id_\cC$ for all $g\in G$. 

In \S\ref{subsection:UniversalExample} below, we give a universal example using Drinfeld doubles of finite groups, which shows that any short exact sequence can be realized as the obstruction \eqref{eq:ShortExactSequence}. 
Using this framework, we provide examples of short exact sequences which do not split, but the induced action of $G$ permutes anyons.
We then include some interesting scenarios where symmetry must be either automatically categorically preserved or broken in \S\ref{sec:AutomaticallyPreservedOrBroken}.


\subsection{Landau theory}
\label{sec:LandauTheory}

The Landau theory of symmetry breaking phase transitions is the simplest example. 
We consider the continuous phase transition between two `trivial' phases with no anyon excitations, with symmetry group $G$ which contains a subgroup $H\subset G$. 
If both phases are gapped, they both correspond to the trivial MTC $\cC_{\text{triv}}=\Vec$ within the categorical framework, where all gapped excitations are local quasiparticles (the trivial object $1$) obeying Bose statistics. 
The associated continuous phase transition between these two phases, described by the Landau-Ginzberg-Wilson paradigm, is driven by condensation of these local bosonic quasiparticles, which corresponds to the trivial connected \'{e}tale algebra $A=1$. 

Since there are no nontrivial anyons involved here, the first obstruction always vanishes:
\begin{equation*}
g(A)\simeq A
\qquad\qquad
\forall g\in G. 
\end{equation*}
Since $\Aut_\cC(A)=1$ in these trivial phases, the second obstruction also vanishes, since the following short exact sequence
\begin{equation*}
\begin{tikzcd}
1 \arrow[r] 
& 
1 \arrow[r, "\iota"] 
& 
H \arrow[r, "\pi"] 
& 
H \arrow[r]
& 
1
\end{tikzcd}
\qquad\qquad
\forall H\subseteq G
\end{equation*}
always splits for any subgroup $H\subset G$. 
This indicates that a trivial phase preserving symmetry group $G$ can be driven into another trivial phase whose symmetry group can be any subgroup $H$ of $G$, via a continuous phase transition from boson condensation. 
This is exactly the symmetry-breaking phase transitions of Landau, described by Landau-Ginzburg theory of fluctuating local order parameters defined in the manifold $G/H$ \cite{Mermin1979}.

Even without anyons and topological orders, the presence of symmetry can also lead to many distinct `trivial' phases, i.e.\ symmetry protected topological (SPT) phases. 
SPT phases can be viewed as a special case of SETO when the corresoponding MTC is trivial. 
Within our categorical framework, the symmetry breaking rules for a trivial MTC do not depend on its SPT classification. 
In other words, for a trivial MTC, its symmetry group $G$ can break down to any subgroup $H\subset G$, independent of the SPT phase to which it belongs.

\subsection{Toric code}
\label{sec:ToricCode}

The toric code \cite{MR1951039} is the simplest example of a topological order in a condensed matter system. 
Its MTC has 4 simple objects
\begin{equation*}
\Irr(\cC_\text{Toric~Code})=\{1,e,m,\epsilon:=e\otimes m\}=\{1,e\}\times\{1,m\},
\end{equation*}
and its modular $S$ and $T$ matrices are given by
\begin{equation*}
S=
\frac12\begin{pmatrix}
1&1&1&1
\\
1&1&-1&-1
\\
1&-1&1&-1
\\
1&-1&-1&1
\end{pmatrix}
\qquad\qquad
T=
\begin{pmatrix}
1&&&
\\
&1&&
\\
&&1&
\\
&&&-1
\end{pmatrix}.
\end{equation*}
Clearly the $e$ and $m$ particles obey Bose statistics, while their composite $\epsilon=e\otimes m$ obeys Fermi statistics. 

Below, we discuss 3 examples involving the toric code preserving a $G=\bbZ_2=\langle g\rangle$ symmetry, where we always choose to condense the boson $e$, corresponding to the connected \'{e}tale algebra $A=1\oplus e$. 
We discuss a few distinct $G=\bbZ_2$ symmetry enriched toric codes, which realize the first obstruction in Example \ref{ex:ToricCode1stObstruction}, the second obstruction in Example \ref{ex:ToricCode2ndObstruction}, and the absence of any obstructions in Example \ref{ex:ToricCodeNoObstruction}.

\begin{ex}
\label{ex:ToricCode1stObstruction}
Consider the $\bbZ_2$ symmetry enriched toric code where the $e$ and $m$ particles are permuted by the action of $\bbZ_2$ symmetry generator $g$:
\begin{equation*}
e\overset{g}\longleftrightarrow m.
\end{equation*}
In this case, the symmetry action $g$ does not preserve the \'etale algebra $A=1\oplus e$, so symmetry is broken at the first level.
This means in a $\bbZ_2$ symmetry enriched toric code where the $\bbZ_2$ symmetry permutes $e$ and $m$, condensing the $e$ particle will necessarily break the $\bbZ_2$ symmetry. 
\end{ex}

\begin{ex}
\label{ex:ToricCode2ndObstruction}
Consider a $\bbZ_2$ symmetry enriched toric code with a nontrivial $\bbZ_2$ symmetry fractionalization class \cite{Essin2013,1410.4540,Tarantino2016}. 
Symmetry fractionalization in a generic SETO is classified by the projective action of symmetry on the anyon $a$ in the SETO:
\begin{equation*}
U_g^{(a)}U_h^{(a)}=\omega_a(g,h)U_{gh}^{(a)},
\qquad
\qquad
\omega_a(g,h)=\frac{S_{a,\omega(g,h)}S_{1,1}}{S_{1,a}S_{1,\omega(g,h)}}\in U(1),
\qquad
\qquad
\omega(g,h)\in\text{Inv}(\cC),
\end{equation*}
which is classified by $[\omega]\in H^2(G,\text{Inv}(\cC))$.

We now specialize to the fractionalization class for $G=\bbZ_2$ symmetry generated by $g$ given by
\begin{equation*}
\omega(g,g)=m
\qquad
\Longrightarrow
\qquad
\omega_e(g,g)=2S_{e,m}=-1. 
\end{equation*}
Since the $\bbZ_2$ symmetry considered here does not permute anyons, the first obstruction vanishes. 
The second obstruction is then captured by the short exact sequence
\begin{equation*}
1\to\Aut_{\cC}(A)\simeq\bbZ_2\to\bbZ_4\to G=\bbZ_2\to 1,
\end{equation*}
which does not split.
Hence the second obstruction does not vanish, and the $\bbZ_2$ symmetry must be broken in the toric code with this symmetry fractionalization class if $e$ is condensed. 
Physically, this means if a symmetry acts projectively on the anyon we wish to condense, this symmetry must be broken by the anyon condensation.
\end{ex}


\begin{ex}
\label{ex:ToricCodeNoObstruction}
Finally, we consider the toric code with a different $G=\bbZ_2 = \{1,g\}$ symmetry fractionalization class, where we have
\begin{equation*}
\omega(g,g)=e
\qquad
\Longrightarrow
\qquad
\omega_e(g,g)=2S_{e,e}=1. 
\end{equation*}
Twisting the trivial action by this 2-cocycle gives a $\bbZ_2$-action on $\cZ(\bbZ_2)$. Note that since $e$ centralizes $A=1_\cC\oplus e$, by Section \ref{sec:StableActions} the short exact sequence is

\begin{equation}
\label{eq:ToricCodeExactSequence}
1\to \Aut_\cC(A)\cong\bbZ_2 \to \Aut_{\cC^G}(I(A))\cong\bbZ_2\times\bbZ_2 \to G=\bbZ_2\to 1\,.
\end{equation}

Thus we see in this case, the obstruction vanishes at both levels since \eqref{eq:ToricCodeExactSequence} splits.
This means the $\bbZ_2$ symmetry can be preserved after we condense the $e$ particle and drive the SETO into a trivial gapped phase $\cC_{\text{triv}}=\Vec$ with no anyons. However, as we will show below, there are two inequivalent splittings of this exact sequence, resulting in a $Z_2$-SPT state with protected edge states versus a trivial state with no edge states after condensing $A=1_\cC\oplus e$. 

Since $A=1_\cC\oplus e$, the two automorphism within $\Aut_\cC(A)$ are given by trivial automorphism $\operatorname{id}_{A}$ and the notrivial one $\operatorname{id}_{1_\cC}\oplus -\operatorname{id}_{e}$. Physically the nontrivial automorphism can be understood as braiding one $m$ particle around the $e$ particle being condensed, which gives rise to a minus sign. Therefore we can conveniently label the automorphism group as
\begin{equation*}
\Aut_\cC(A)=\{\operatorname{id}_{A},\operatorname{id}_{1_\cC}\oplus -\operatorname{id}_{e}\}\cong\{1,m\}\cong\bbZ_2.
\end{equation*}

In this case, the corresponding $\bbZ_2$-crossed braided extension is given by $\cF:=\cZ_{\Vec(\bbZ_4)}(\Vec(\bbZ_2))$, the relative center of $\Vec(\bbZ_2)$ in $\Vec(\bbZ_4)$. The associated gauged theory is $\cD:=\cF^{\bbZ_2}\cong \cZ(\Vec(\bbZ_4))$, whose anyons are given by
\begin{equation}
\label{eq:AnyonsInDoubleZ4}
\{1,\alpha,\alpha^2,\alpha^3\}\times\{1,m,m^2,m^3\}.
\end{equation}
The two generating anyons $\alpha$ and $m$ both have a trivial self-braiding, and obey a $\bbZ_4$ fusion rule:
\begin{equation*}
\alpha^4\cong m^4\cong1. 
\end{equation*}
Here, $m$ is exactly the $m$ particle in the toric code before gauging, while $\alpha^2$ corresponds to the $e$ particle before gauging. 

The two inequivalent splittings correspond to the following two choices of $I(A)$:
\begin{equation*}
(A,\lambda_1)\cong I(A)=1_\cD\oplus \alpha^2
\end{equation*}
and
\begin{equation*}
(A,\lambda_2)\cong I(A)=1_\cD\oplus s^2
\qquad\text{where}\qquad
s:=\alpha\otimes m.
\end{equation*}
Both splittings share the same automorphism group
\begin{equation*}
\Aut_{\cC^G}(I(A))
=
\Aut_{\cF^G}(I(A))
=
\Aut_\cD(I(A))
\cong\{1,\alpha\}\times\{1,m\}
\cong
\bbZ_2\times\bbZ_2.
\end{equation*}

We provide the following diagram of the relevant phases, where squiggly arrows denote condensing by the labeled algebra, and other arrows denote gauging the $\bbZ_2$ symmetry.
\begin{equation*}
  \begin{tikzcd}
   & 
   \cC_{A}^\loc=\fdVec 
   \arrow[]{rd}{\text{gauge}} 
   &  
   \\
  \cC=\cZ(\Vec(\bbZ_2)) 
  \arrow[squiggly]{ru}{A=1\oplus e} 
  & 
  \cD^{\mathrm{loc}}_{1_\cD\oplus\alpha^2}\cong\cZ(\Vec(\bbZ_2))
   \arrow[<-]{u}{\text{gauge}}
   & 
   \cD^{\mathrm{loc}}_{1_\cD\oplus s^2} \cong \cZ(\mathrm{Sem})
   \\
   & 
   \cD=\cZ(\Vec(\bbZ_4)) 
   \arrow[squiggly]{u}{(A,\lambda_1)} 
   \arrow[squiggly, swap]{ru}{(A,\lambda_2)}  
   \arrow[<-]{ul}{\text{gauge}} 
   & 
  \end{tikzcd}
\end{equation*}

The two inequivalent splittings of \eqref{eq:ToricCodeExactSequence} have a clear physical difference in the outcome of the anyon condensation \cite{Jiang2017}. The two splittings are characterized by different $g$-defects of the $\bbZ_2$ symmetry group:
\begin{equation*}
X_g^1\cong\alpha
\qquad\qquad
X_g^2\cong s=\alpha\otimes m. 
\end{equation*}
In particular $X_g^1\cong\alpha$ is a boson for the first splitting, and $X_g^2\cong s$ is a semion for the second splitting. 
As a result, they give rise to different $\bbZ_2$-symmetric gapped phases after condensing $A=1_\cC\oplus e$: 
\begin{enumerate}[(1)]
\item 
For the first splitting, each $e$ particle carries the trivial (one-dimensional) linear representation of the $G=\bbZ_2$ symmetry group. 
It results in a trivial $\bbZ_2$-symmetric phase with no edge states, leading to a toric code by gauging the $\bbZ_2$ symmetry
\begin{equation*}
\cD^{\mathrm{loc}}_{1_\cD\oplus\alpha^2}\cong\{1,\alpha\}\times\{1,m^2\}\cong\cZ(\Vec(\bbZ_2)).
\end{equation*}
\item
For the 2nd splitting, each $e$ particle carries the nontrivial (one-dimensional) linear representation of the $G=\bbZ_2$ symmetry group. 
It results in a nontrivial $\bbZ_2$-SPT phase with protected edge states, leading to a double semion theory after gauging the $\bbZ_2$ symmetry
\begin{equation*}
\cD^{\mathrm{loc}}_{1_\cD\oplus s^2}\cong\{1,s\}\times\{1,m^2\}\cong\text{double semion}.
\end{equation*}

\end{enumerate}



\end{ex}


\subsection{General stable actions}
\label{sec:StableActions}

We now consider general stable actions, generalizing Example \ref{ex:ToricCodeNoObstruction}. Let $\cC$ be a fusion category. 
A categorical action $\uG\rightarrow \uEqTens(\cC)$ is called \textit{stable} if each $g\in G$ acts as the identity monoidal functor. 
Such an action is specified up to equivalence by a 2-cocycle $\omega\in Z^{2}(G, \Aut_{\otimes}(\id_\cC))$, where $\Aut_{\otimes}(\id_\cC)=\pi_1(\uEqTens(\cC))$ is the monoidal natural automorphisms of the identity functor, and the action of $G$ on $\Aut_{\otimes}(\id_\cC)$ is trivial. 
The 2-cocycle describes the tensorator simply by $\mu^{a}_{g,h}=\omega(g,h)\id_{a}$ for any $a\in \cC$.

Clearly $g(A)=A$ as algebras for all $g\in G$. 
When $\cC$ is nondegenerately braided,  $\pi_2(\uEqBr(\cC))\cong \Inv(\cC)$ as `ring operators' as in Remark \ref{rem:Pi2OfEqBr(C)}.
Note that each ring operator gives an element in the center $Z(\Aut_{\cC}(A))$, and thus $\omega$ yields a 2-cocycle $\widetilde{\omega}\in Z^{2}(G, Z(\Aut_{\cC}(A)))$. 

Directly from the definitions, it is clear that we can find an equivariant algebra structure on $A$ precisely when we have a $\widetilde{\omega}$-projective homomorphism $\psi: G\rightarrow \Aut_{\cC}(A)$. 
The short exact sequence in this case is
$$
\begin{tikzcd}
1 \arrow[r] 
& 
\Aut_{\cC}(A) \arrow[r, "i"] 
& 
\Aut_{\cC}(A)\times_{\widetilde{\omega}} G \arrow[r, "p"] 
& 
G \arrow[r]
& 
1
\end{tikzcd}
$$
where $\Aut_{\cC}(A)\times_{\widetilde{\omega}} G$ is the twisted direct product. 
This sequence has a splitting precisely when we have an $\tilde{\omega}$-projective homomorphism $\psi: G\rightarrow \Aut_{\cC}(A)$. In this case the splitting is given by $g\mapsto (\psi(g), g)\in \Aut_\cC(A)\times_{\widetilde{\omega}}G$.

\subsection{Universal example from an exact sequence of groups}
\label{subsection:UniversalExample}

Suppose we have a short exact sequence of finite groups as in \eqref{eq:ShortExactSequenceOfGroups}:
$$
\begin{tikzcd}
1 \arrow[r] 
& 
N \arrow[r, "i"] 
& 
E \arrow[r, "p"] 
& 
G \arrow[r]
& 
1.
\end{tikzcd}
$$
Consider the non-degenerately braided fusion category $\cC:=\cZ(\Vec(N))$.
Since $\Vec(E)$ is a $G$-extension of $\Vec(N)$, by \cite[Thm~3.3]{MR2587410}, the relative center $\cF:=\cZ_{\Vec(E)}(\Vec(N))$ is a $G$-crossed braided extension of $\cZ(\Vec(N))$.
Furthermore, the equivariantization $\cF^{G}=\cZ_{\Vec(H)}(\Vec(N))^{G}$ is braided tensor equivalent to $\cZ(\Vec(E))$, which is a gauging of $\cZ(\Vec(N))$ by $G$. The inverse construction to gauging is taking local modules of function algebras.
So condensing $\cO(G)\in \Rep(E)\subseteq \cZ(\Vec(E))$, we have $\cZ(\Vec(E))^{\loc}_{\cO(G)}$ is braided tensor equivalent to $\cZ(\Vec(N))$.

Recall that we have an adjoint pair of induction and restriction functors
$$
\begin{tikzcd}
\Rep(E)
\ar[shift left]{rr}{\Res^E_N}
&&
\Rep(N)
\ar[shift left]{ll}{\Ind_N^E}.
\end{tikzcd}
$$
Since $\Res_{N}^E$ is a braided tensor functor, we have a braided equivalence
$$
\Rep(N)\cong \Rep(E)_{\Ind_N^E(\triv_N)} = \Rep(E)_{\cO(E/N)} = \Rep(E)_{\cO(G)}
 .
$$
This equivalence identifies $\cO(E/N) = \Ind_{N}^E(\triv_N)$ and $\cO(E) = \Ind_N^E(\cO(N))$.
(See also \cite[Ex.~4.15.3]{MR3242743}.)



Set $A:=\cO(N)\in \Rep(N)\subset \cZ(\Vec(N))=\cC$.
As above, 
$$
I(A) = \Ind_{N}^E(\cO(N)) = \cO(E) \in \Rep(E)\subset \cC^G \subset \cF^G \cong \cZ(\Vec(E)).
$$
Hence we have isomorphisms $\Aut_\cC(A)\cong N$ and $\Aut_{\cC^G}(I(A))=\Aut_{\cF^G}(I(A))  \cong E$.
Explicitly, these isomorphisms are both instances of the fact that $K =\Aut_{\Rep(K)}(\cO(K))$ for a finite group $K$.
Indeed, since $\cO(K)$ is a left $K$-module via $(k\cdot f)(\ell) := f(k^{-1}\ell)$, and since left multiplication commutes with right multiplication, we get a $K$-equivariant map $\theta_k \in \Aut_{\Rep(K)}(\cO(K))$ by $(\theta_k f)(\ell) = f(\ell k)$.


We now check that the following diagram commutes:
$$
\begin{tikzcd}
1
\ar[r,""]
&
\Aut_\cC(A) 
\ar[r,"\iota"]
\ar[d,<->,"\cong"]
&
\Aut_{\cC^G}(I(A))
\ar[r,"\pi"]
\ar[d,<->,"\cong"]
&
G
\ar[r,""]
\ar[d,equal]
&
1
\\
1 \arrow[r] 
& 
N \arrow[r, "i"] 
& 
E \arrow[r, "p"] 
& 
G \arrow[r]
& 
1.
\end{tikzcd}
$$
For the square on the left, notice that as left $N$-modules, we may identify
\begin{equation}
\label{eq:IdentificationOfI(A)WithO(Ng)}
I(A) = \cO(E) \cong \bigoplus_{g\in G} \cO(Ng^{-1})
=
\bigoplus_{g\in G} g(\cO(N)),
\end{equation}
and we may identify $\cO(Ng^{-1})$ with $g(\cO(N))$ as algebra objects in $\Rep(N)$.
Recall that for $\theta_n \in \Aut_{\Rep(N)}(\cO(N))=N$, $\iota(\theta_n) = \bigoplus_{g\in G} g(\theta_n)$ from Proposition \ref{prop:InjectiveGroupHom}.
Under the identification \eqref{eq:IdentificationOfI(A)WithO(Ng)}, the action of $g(\theta_n)$ on $g(\cO(N)) \cong \cO(Ng^{-1})$ is via $\theta_{gng^{-1}}\in \Aut(\cO(Ng))$.
Hence the left square commutes.

Now in Proposition \ref{prop:DefineSurjectiveGroupHom}, we defined $\pi$ on $f\in \Aut_{\cC^G}(I(A))$
by $\pi(f)$ is the unique $g\in G$ such that 
$0\neq f_{e,g}: g(A) \to A$.
Under the identification $I(A) = \cO(E)$ and $g(\cO(N))=\cO(Ng^{-1})$, for every $x\in E=\Aut_{\Rep(E)}(\cO(E))$, there is a unique $g\in G$ such that $x\in Ng$.
So if $f\in \cO(E)$ is supported on $Ng^{-1}$, then $\theta_x f$ is supported on $N$.
Mapping $x\in E$ to the unique $g\in G$ such that $x\in Ng$ is exactly the quotient map $E \to G$.
Hence the square on the right commutes.

\begin{rem}
Notice that we have a canonical isomorphism $\pi_2(\uEqBr(\Rep(N)) \cong Z(N)$.
Indeed, since $\Rep(N)_{\rm ad} = N/Z(N)$ \cite[Ex.~4.14.6]{MR3242743} and $\Rep(N)$ is faithfully graded by the dual group $Z(N)^{\widehat{\,\,}}$, the characters of the grading group of $\Rep(N)$ are canonically identified with $Z(N)^{\widehat{\widehat{\,\,}}}\cong Z(N)$.
Given a homomorphism $G\rightarrow \Out(N)$, one obtains a 3-cohomology class $o_{3}\in H^{3}(G, Z(N))$, which vanishes if and only if there exists a short exact sequence $1\rightarrow N\rightarrow E\rightarrow G\rightarrow 1$ for some $E$ which implements this outer action. 
This 3-cohomology class agrees with the $o_{3}$ obstruction associated to $G\to \EqBr(\Rep(N))$.
\end{rem}

\begin{ex}
In previous examples with an obstruction in the second level, we always considered a trivial action of the symmetry which does not permute anyons. 
Below, making use the universal example from a short exact sequence, we construct an explicit example where (i) the symmetry permutes anyons, (ii) the first obstruction vanishes, and (iii) the second obstruction does not vanish.

We consider a $G=\bbZ_2$ SETO with
\begin{equation}
\cC=\cZ(\Vec(\bbZ_3))\boxtimes\cZ(\Vec(\bbZ_2)).
\end{equation}
Physically, $\cC$ corresponds to one layer of $\bbZ_3$ gauge theory $\cZ(\Vec(\bbZ_3))$ stacked on top of the toric code $\cZ(\Vec(\bbZ_2))$ (or $\bbZ_2$ gauge theory). 
We label the gauge charge of $\bbZ_3$ gauge theory as $\tilde e$, and its gauge flux as $\tilde m$, satisfying fusion rules
\begin{equation*}
\tilde e\otimes\tilde e\otimes\tilde e\equiv\tilde e^3\cong\tilde m^3\cong 1.
\end{equation*}
We still use $e$ and $m$ to label the gauge charge and gauge flux in toric code $\cZ(\Vec(\bbZ_2))$. 
The $\bbZ_2$ symmetry $g$ permutes anyons $\tilde e$ and ${\tilde e}^2\cong \tilde{e}^{-1}$, and belongs to a nontivial symmetry fractionalziation class
\begin{equation*}
\omega(g,g)=m
\qquad
\Longrightarrow
\qquad
\omega_e(g,g)=-1.
\end{equation*}
We condense the connected \'etale algebra
\begin{equation*}
A:=
\cO(\bbZ_3)\boxtimes \cO(\bbZ_2)
=
1\oplus \tilde e\oplus\tilde e^2\oplus e\oplus (e\otimes\tilde e)\oplus (e\otimes\tilde e^2).
\end{equation*}
Clearly $g(A)\cong A$, and the first obstruction vanishes.
At the second level, since $\Aut_\cC(A)\cong\bbZ_2\times\bbZ_3\simeq\bbZ_6$, the short exact sequence \eqref{eq:ShortExactSequence} is given by
\begin{equation*}
1\to \Aut_\cC(A)\simeq\bbZ_6 \to \operatorname{Dic}_{12} \cong \bbZ_3 \rtimes \bbZ_4 \to G=\bbZ_2 \to 1,
\end{equation*}
which does not split, leading to an obstruction at the second level. 
Physically, since the $\bbZ_2$ symmetry $g$ acts projectively on the anyon $e$, condensing $e\otimes\tilde e$ will also condense $e$, and therefore the $\bbZ_2$ symmetry must be broken.
\end{ex}

\begin{rem}
The universal example can be modified by introducing a cocycle on $E$ as follows.
Given $\omega\in Z^3(E, \bbC^\times)$, we set $\cC:=\cZ(\Vec(N,\omega|_N))$.
Note that $\cC$ has a $G$-crossed braided extension $\cF$ such that $\cF^G$ is braided equivalent to $\cZ(\Vec(E,\omega))$. 
Indeed, the discussion at the beginning of this section only uses the fact   that there is a canonical Tannakian subcategory   $\Rep(H)\subset \cZ(\Vec(H,\omega))$ for $H=N,E$, and the exact sequence does not change under the presence of cocycles.
We also note the relation to the Lyndon--Hochschild--Serre spectral sequence, see \cite[Appendix by E.\ Meir]{MR2677836}.
  
If the obstruction vanishes, one gets, for every splitting $\sigma:G\hookrightarrow E$, a $G$-graded extension of  the condensed theory $\cC_{\cO(N)}^\loc\cong\Vec$ which is necessarily given by $\Vec(G,\upsilon)$ for some $[\upsilon]\in H^3(G,\bbC^\times)$ given by the pull back $[\upsilon]=\sigma^\ast[\omega]=[\omega\circ \sigma^{\times 3}]$ along the splitting $\sigma:G\hookrightarrow E$.

Now it may be the case that given a fixed exact sequence and a fixed $\omega\in Z^3(E,\bbC^\times)$, two inequivalent splittings yield two non-cohomologous cocycles $\upsilon \in Z^3(G, \bbC^\times)$!
We would like to thank Tian Lan for explicitly pointing this phenomenon out to us, which we saw briefly in Example \ref{ex:ToricCodeNoObstruction} above.
(See also \cite{1801.01210}.)
We now provide further details on this example below.
\end{rem}

\begin{ex}
Let us consider the special case of 
the exact sequence 
$$
\begin{tikzcd}
1 \arrow[r] 
& 
\bbZ_2 \arrow[r, "\iota_1"] 
& 
\bbZ_2\times\bbZ_2\arrow[r, "\pi_2"] 
& 
\bbZ_2 \arrow[r]
& 
1,
\end{tikzcd}
$$
where $\iota_1$ is inclusion into the first factor and $\pi_2$ is the projection onto the second factor.
Now consider $\omega\in Z^3(\bbZ_2\times\bbZ_2,\bbC^\times)$ with one non-trivial element
$\omega((1,1),(1,1),(1,1))=-1$.
There are two splittings, which give exactly the two cohomology classes of $H^3(\bbZ_2,\bbC^\times)$.

Indeed, this example is the same as Example \ref{ex:ToricCodeNoObstruction} above, as there is a braided tensor equivalence $\cZ(\Vec(\bbZ_2\times\bbZ_2, \omega))\cong \cZ(\Vec(\bbZ_4))$ \cite[Ex.~4.10]{MR2362670}!
This can also be seen from the following table of twists in $\cZ(\Vec(\bbZ_4))$ using the anyon labels from \eqref{eq:AnyonsInDoubleZ4} above
$$
\begin{array}{c|rrrr}
 \theta& 1 & m & m^2 & m^3\\
 \hline
 1 & 1 & 1& 1 & 1\\
 \alpha & 1 & i & -1 & -i \\
 \alpha^2 & 1 & -1 & 1 & -1\\
 \alpha^3 & 1 & -i & -1 & i
\end{array}
$$
together with the following table of connected \'{e}tale algebras in $\cZ(\Vec(\bbZ_4))$:
$$
\begin{array}{l|c|c}
  A &\cZ(\Vec(\bbZ_4))_A^\loc & \Aut_{\cZ(\Vec(\bbZ_4))}(A)\\
  \hline
  1 & \cZ(\Vec(\bbZ_4)) & 1 \\
  1\oplus \alpha & \cZ(\Vec(\bbZ_2)) & \bbZ_2 \\
  1\oplus m & \cZ(\Vec(\bbZ_2)) & \bbZ_2 \\
  1\oplus \alpha^2m^2 & \cZ(\mathrm{Sem}) & \bbZ_2 \\
  1\oplus \alpha \oplus \alpha^2\oplus \alpha^3 & \Vec & \bbZ_4 \\
  1\oplus m\oplus m^2\oplus m^3 & \Vec & \bbZ_4 \\
  1\oplus \alpha^2\oplus m^2\oplus \alpha^2m^2 & \Vec & \bbZ_2\times\bbZ_2
\end{array}
$$
\end{ex}

\subsection{Symmetry which is automatically categorically preserved or broken}
\label{sec:AutomaticallyPreservedOrBroken}

In the next three sections, we present some interesting cases where symmetry is automatically categorically preserved or broken.
In \S\ref{sec:Aut(A)Trivial}, we consider when $\Aut_\cC(A)$ is trivial, so that the exact sequence \eqref{eq:ShortExactSequence} is always an isomorphism.
In \S\ref{sec:FirstLevelBroken}, we consider the case when $\Aut_\cC(A)\cong \Aut_{\cC^G}(I(A))$, so that no non-trivial $G$-action may satisfy the first obstruction.
Finally, in \S\ref{sec:MultilayerPhase}, we consider a multilayer phase where the algebra lies in one level and the action happens on another, so that the exact sequence \eqref{eq:ShortExactSequence} trivially splits.

\subsubsection{\'Etale algebras with trivial automorphism group}
\label{sec:Aut(A)Trivial}

Suppose $\cC$ is a non-degenerately braided fusion category and 
$A\in\cC$ is a connected \'etale algebra with trivial automorphism group $\Aut_{\cC}(A)\cong 1$. 
We assume there is a categorical action $\underline{\rho}:\underline{G}\to\uEqBrStab(A)\subset \uEqBr(\cC)$. 
This implies that
the exact sequence \eqref{eq:ShortExactSequence} necessarily collapses to
\begin{equation*}
\begin{tikzcd}
1
\ar[r,"\iota"]
&
\Aut_{\cC^G}(I(A))
\ar[r,"\pi"]
&
G
\ar[r,""]
&
1.
\end{tikzcd}
\end{equation*}
In other words, 
$\pi:\Aut_{\cC^G}(A)\to G$ is an isomorphism. 
In particular, there is a unique lifting
$\sigma=\pi^{-1}$
as in \S\ref{sec:LandauTheory}.

We now show that non-trivial such examples indeed exist. 
First, we note the following sufficient condition for $\Aut_{\cC}(A)$ to be trivial, which  essentially follows from the argument in \cite[\S6 Type $E_6$]{MR1936496}.

\begin{prop} 
\label{prop:SufficientTrivialAutomorphismGroup}
Suppose $\cC$ is a fusion category and $X\in \cC$ is a self-dual object with $\dim(\cC(X^{\otimes 2}\to X))=1$.
If $A:= 1\oplus X$ admits the structure of a separable algebra object, then the algebra structure is unique up to unique isomorphism.
In particular, $\Aut_{\cC}(A)$ is trivial.
\end{prop}


The following is a well-known example of a non-trivial \'etale algebra $A$ which fulfills the 
hypothesis of Proposition \ref{prop:SufficientTrivialAutomorphismGroup}.
\begin{ex}
\label{ex:GHJ}
Let $\cC(\SU(2)_{10})$ with simple objects $([0],[\frac{1}{2}],[1],\dots, [5])$ and fusion rules 
$$
  [i]\otimes [j]\cong \bigoplus^{i+j}_{\substack{k=|i-j|\\i+j+k\leq 10}} [k],
$$
and consider the unique connected \'{e}tale algebra 
$A_{\mathrm{GHJ}} := [0]\oplus [3]$ \cite[\S6 Type $E_6$]{MR1936496} (see also \cite{MR999799}).
By Proposition \ref{prop:SufficientTrivialAutomorphismGroup}, $\Aut(A_{\mathrm{GHJ}})$
is trivial. 
\end{ex}

Now $\cE_6:=\cC(\SU(2)_{10})_{A_{\mathrm{GHJ}}}$ is an $E_6$-category, i.e.\
the fusion matrix of the (necessarily simple) free $A$-module $[\frac12]\otimes A$ is the adjancency matrix of the Dynkin diagram $E_6$.
The condensation $\cC_{A_{\mathrm{GHJ}}}^\loc 
$ is braided equivalent to Ising category $\cC(\Spin(5)_1)$ 
with $T=\diag(1,\exp(\frac{10\pi i}{16}),-1)$.


Recall from \cite[Cor.~4.9 and Ex.~5.1]{MR1815993} (see also \cite[Corollary 3.30]{MR3039775}) that we have braided equivalences
$$
\cZ(\cE_6)\cong \cC(\SU(2)_{10}) \boxtimes \overline{\cC_{A_{\mathrm{GHJ}}}^{\loc}}
\cong \cC(\SU(2)_{10}) \boxtimes \overline{\cC(\Spin(5)_{1})}\cong 
\cC(\SU(2)_{10}) \boxtimes \cC(\Spin(11)_{1})\,.
$$
We denote by $\frac{1}{2}\cE_6$ the adjoint subcategory of 
$\cE_6$, which has fusion rules of the even part of $E_6$ \cite{MR1832764}:
$$
  \alpha^2\cong  1
  \qquad\qquad
  \sigma \otimes \alpha\cong \alpha\otimes \sigma\cong \sigma
  \qquad\qquad
  \sigma\otimes \sigma \cong 1\oplus \alpha \oplus 2\cdot \sigma
, 
$$
and note that $\cE_6$ is 
a $\bbZ/2$-extension of $\frac{1}{2}\cE_6$.

Let $\cC$ be the modular tensor category $\cZ(\frac{1}{2}\cE_6)$ 
which has 10 simple objects and the modular data is given in \cite{MR1832764,MR2468378}.
By \cite[Thm~3.3]{MR2587410} it follows that $\cC=\cZ(\frac{1}{2}\cE_6)$ has a $\bbZ_2$-crossed braided extension $\cF=\cZ_{\cE_6}(\frac{1}{2}\cE_6)$. 
Let $\cD=\cC(\SU(2)_{10})\boxtimes \cC(\Spin(11)_1)$.
As noted in \cite{MR3462031} we can also realize $\cF$ 
as $\cD_{1\oplus \beta}$, where $\beta$ is the unique $\bbZ_2$-boson in $\cD$.
In other words,
$\cC$ is the $\bbZ_2$ condensation of $\cD=\cC(\SU(2)_{10})\boxtimes \cC(\Spin(11)_1)$.

\begin{prop} The center 
$\cC=\cZ(\frac{1}{2}\cE_6)$ of the even part of $E_6$  has a unique \'etale algebra 
$A\cong 1\oplus X$ for a non-invertible object $X$ of dimension $2+\sqrt{3}$, such that the $\cF$ with associated $\bbZ_2$-action $\underline{\rho}: \underline{G}
\to \uEqBr(\cC)$ lands in $\uEqBrStab(A)$.
\end{prop}
\begin{proof}
The adjoint functor $I:\cE_6\to \cZ(\cE_6)\cong\cD$ of the forgetful functor 
 gives a Lagrangian algebra $L'=I(1)$ in $\cD$.  Subalgebras of $L'$ correspond to full subcategories of $\frac{1}{2}\cE_6$ \cite[Theorem 4.10]{MR3039775}. 
The maximal pointed subcategory $\cE_6^\times\subset \cE_6$ correspond to an \'etale algebra
$A'\in\cD$.
Note that the \'etale algebra $B=1\oplus\beta$ corresponds to the adjoint subcategory $\frac{1}{2}\cE_6\subset \cE_6$ and $A'$ gives an \'etale 
algebra in $\cC\cong \cD_B^\loc$ which we denote by $A$. 
A simple dimension argument gives $\dim(A)=3+\sqrt{3}$
and it follows from the modular data of $\cC$ that $A\cong 1\oplus X$ as claimed.

To see that the action stabilizes, we note that $g(X)\cong X$, since the object $X$ splits into two inequivalent objects in $\cD\cong\cF^{\bbZ_2}$, which can be read of the modular invariant which has been computed in
\cite[Example 4.13]{MR3462031}.
Thus $g$ fixes $A$ as an object, but $A$ has a unique algebra structure by Proposition \ref{prop:SufficientTrivialAutomorphismGroup}.
\end{proof}

This implies $\Aut_\cC(A)$ is trivial and that there is a unique equivariant 
algebra $(A,\lambda)\in \cF^{\bbZ_2}$.
Again it can be checked that the only possibility 
as an object (and hence as an algebra using Proposition \ref{prop:SufficientTrivialAutomorphismGroup})
is that $(A,\lambda)$ is isomorphic to $A_{\mathrm{GHJ}}\boxtimes 1$ in $\cD$.
It further follows that 
$$
  (\cF_A^{\bbZ_2-\loc})^{\bbZ_2}\cong
  \cD_{(A,\lambda)}^\loc\cong 
  \cC(\Spin(5)_1)\boxtimes\cC(\Spin(11)_1)
$$
which is a center of an Ising category. 
The condensed theory $\cC_A^\loc\cong\cC(\Spin(16)_1)$
is the toric code $\cZ(\bbZ_2)$. We summarize with the following diagram.
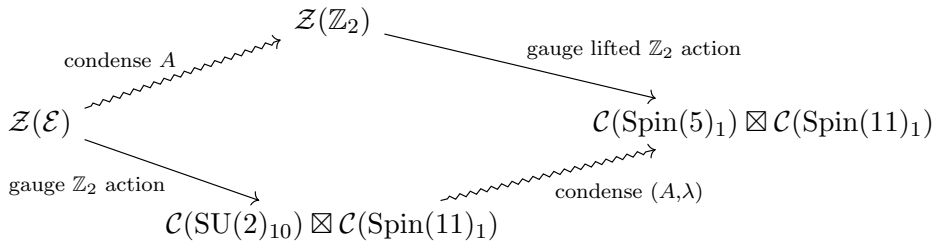
\begin{equation*}
\begin{tikzcd}
  & \cZ(\bbZ_2) \arrow[rd, "\text{gauge lifted }\bbZ_2\text{ action}"] &  
  \\
  \cZ(\cE) \arrow[rd, "\text{gauge }\bbZ_2\text{ action}"'] \arrow[ru,squiggly, "\text{condense }A"] &  & 
\cC(\Spin(5)_1)\boxtimes\cC(\Spin(11)_1)
\\
  & \cC(\SU(2)_{10})\boxtimes\cC(\Spin(11)_1) \arrow[ru,squiggly,"\text{condense }{(A,\lambda)}"']
  &
\end{tikzcd}
\end{equation*}

\subsubsection{First level  broken symmetry}
\label{sec:FirstLevelBroken}

Suppose $\cC$ is a non-degenerately braided fusion category and $A\in \cC$ is a connected \'etale algebra. 
Suppose we have a non-trivial finite group $G$ and a categorical action $\underline{G}\to \uEqBr(\cC)$.
If $g(A)\not\cong A$ as algebras for all $g\in G$ the symmetry must be completely broken. 


In particular, suppose 
that we have a group isomorphism $\Aut_\cC(A) \cong \Aut_{\cC^G}(I(A))$.
Then $g(A)\ncong A$ as an algebra for any $g\in G$; otherwise, the existence of our short exact sequence leads to a contradiction.
Hence the symmetry is completely broken.

\begin{ex}
\label{ex:Metaplectic}
For an explicit example, we take $\cC :=\cZ(\Vec(\bbZ_3))\cong \Vec(\bbZ_3, q) \boxtimes \Vec(\bbZ_3, \overline{q})$ where $q(n)=\exp(2\pi i n^2/3)$ is a quadratic form on $\bbZ_3$, and $A:= \bigoplus_{g\in \bbZ_3} g\boxtimes g^{-1}$.
Let $\underline{\bbZ_2}\to \uEqBr(\Vec(\bbZ_3, q))$ be the particle-hole symmetry \cite{MR3555361}, which gives a categorical $\bbZ_2$-action on $\cC$. 
Indeed, there are two equivalence classes of $\bbZ_2$-crossed braided extension of $\cC$ which are of the form 
$\cM\boxtimes \Vec(\bbZ_3,\overline{q})$ for $\cM$ a $\bbZ_3$-metaplectic category \cite{MR3541678}.
Let us denote by $\alpha$ the non-trivial element of $\uEqBr(\cC)$ with $\alpha(g\boxtimes k)\cong g^{-1}\boxtimes k$.
This action maps $\alpha(g\boxtimes g^{-1})\cong g^{-1}\boxtimes g$, and thus $\alpha$ does not even fix $A$ as an object. 

We claim that $\Aut_\cC(A) \cong \Aut_{\cC^G}(I(A))$.
First, note that $\Aut_\cC(A) =\bbZ_3$.
Next, it is straightforward to show that 
$$
I(A)\cong (1\boxtimes 1)\oplus (h\boxtimes 1)\oplus (X\boxtimes g)\oplus (X\boxtimes g^{-1})
$$ 
where $h\in (\Vec(\bbZ_3,q)^{\bbZ_2}$ is a boson, $X\in (\Vec(\bbZ_3,q)^{\bbZ_2}$ has dimension $2$, and $g$ is a generator of $\Vec(\bbZ_3,\overline q)$. 
Hence $\dim_{\cC^G}(I(A))=6$, and $|\Aut_{\cC^G}(I(A))|\leq 6$ \cite{DaBi2018}.
Assume for contradiction that $|\Aut_{\cC^G}(I(A))|=6$ so that 
$I(A)\cong\cO(H)$ for some group $H$ of order 6, see \cite{DaBi2018}.
But
$$
\cO(H)\cong\bigoplus_{\pi\in\Irr(\Rep(H))}\dim(\pi) \cdot \pi
\ncong
I(A),
$$
a contradiction.
Hence $|\Aut_{\cC^G}(I(A))|<6$.
Since we always have an injection $\Aut_\cC(A)=\bbZ_3\hookrightarrow \Aut_{\cC^G}(I(A))$, this injection must be an isomorphism.
\end{ex}

\begin{ex}
\label{ex:S3}
We now consider the other equivalence class of $\bbZ_2$-crossed braided extension of $\cC$, which is in stark contrast to Example \ref{ex:Metaplectic}. 
This other extension comes from the universal example from \S\ref{subsection:UniversalExample} corresponding to the exact sequence 
$$
\begin{tikzcd}
1 \arrow[r] 
& 
\bbZ_3 \arrow[r, "i"] 
& 
\bbZ_3\rtimes\bbZ_2\cong S_3 \arrow[r, "p"] 
& 
\bbZ_2 \arrow[r]
& 
1
\end{tikzcd}
$$
which permits three different splittings.
\end{ex}
As in \cite{1803.04949}, Examples \ref{ex:Metaplectic} and \ref{ex:S3} can be generalized to any finite abelian group of odd order.

\begin{prob}
Find an explicit example of a non-degenerate braided fusion category $\cC$, a connected \'etale algebra $A$, and a categorical action $\underline{\rho}:\underline{G}\to \uEqBr(\cC)$ with $\Aut_\cC(A)\cong \Aut_{\cC^G}(I(A))$ 
such that $g(A)\cong A$ \emph{as an object alone}, but not necessarily as an algebra.
\end{prob}

\subsubsection{Trivially unbroken symmetry}
\label{sec:MultilayerPhase}

\begin{ex}
Suppose $\cC_1,\cC_2$ are non-degenerately braided fusion categories, and define $\cC:=\cC_1\boxtimes \cC_2$.
Suppose we have a categorical action $\uG\to \uEqBr(\cC_2)$, which gives us a categorical action $\uG \to \uEqBr(\cC)$.
Suppose now $A_1 \in \cC_1$ is a connected \'{e}tale algebra, and consider $A=A_1\boxtimes 1_{\cC_2}\in \cC$.
Then $\Aut_{\cC^G}(I(A))\cong \Aut_\cC(A)\times G \cong \Aut_{\cC_1}(A_1)\times G$, and the exact sequence \eqref{eq:ShortExactSequence} clearly splits.
\end{ex}

\section{Application to algebraic quantum field theory}
\label{sec:AQFT}

A \emph{conformal net} $\cB=\{\cB(I)\}_{I\in\cI}$ is a family of von Neumann algebras 
on a common Hilbert space $H$ indexed by the set $\cI$ of proper open intervals on the unit circle $S^1$ together with a cyclic and separating unit vector $\Omega\in H$ and a projective unitary representation of $\mathop{\mathrm{Diff}}_+(S^1)$, the group of orientation preserving diffeomorphisms of $S^1$, fulfilling certain axioms, see e.g.\  \cite{MR2007172} for details.

If $\cB$ is a so-called \emph{rational} conformal net, then $\cC:=\Rep^I(\cB)$ is a unitary modular tensor category \cite{MR1838752}.
A conformal net $\cB$ has a group of \emph{global gauge transformations} $\Aut(\cB)$.
There is a one-to-one correspondence \cite{MR1332979} between local irreducible extensions $\cA\supset \cB$ and connected \'etale ${\rm C^*}$ Frobenius algebras 
$A\in\Rep(\cB)$, and $\Rep(\cA)$ is braided tensor equivalent to $\cC^{\mathrm{loc}}_{A}$ \cite{MR3424476}.
Thus local extensions correspond to condensation of \'etale algebras.
Let us denote the global gauge transformations of $\cA$ fixing 
$\cB$ pointwise by
\begin{equation*}
  \Aut(\cA|\cB)=\set{\alpha\in \Aut(\cA)}{\alpha(b)=b \text{ for all }b\in\cB(I)\text{ and }I\in\cI}\,,
\end{equation*}
which is a finite group since the index $[\cA:\cB]=\dim_{\Rep(\cB)}(A)$ is finite.
By \cite{MR2007172}, $|\Aut(\cA|\cB)|\leq [\cA:\cB]$.
\begin{lem}
  \label{lem:AutomorphismsOfCNets}
  Under the above correspondence between $A\in\Rep(\cB)$ and $\cA\supset \cB$, we have a correspondence
$\Aut(\cA|\cB)\longleftrightarrow\Aut_{\Rep(\cB)}(A)$.
\end{lem}
\begin{proof}
Let $H$ be a Hilbert space, let $N\subset M\subset B(H)$ be von Neumann algebras, and let $\Omega\in H$ be a cyclic and separating vector.
We denote
\begin{equation*}
  \Aut(M|N,\Omega)=\set{\alpha\in\Aut(M)}{\alpha(n)=n 
  \text{ and } \langle\Omega|\alpha(m)\Omega\rangle=\langle\Omega|m\Omega\rangle
  \text{ for all }n\in N, m\in M}\,.
\end{equation*}
Fix a proper interval $I\in\cI$. 
By \cite{MR3595480} 
we have $\Aut(\cA(I)|\cB(I),\Omega)=\Aut(\cA|\cB)$.
  We may assume that $A\in\Rep^I(\cA)\subset \End(\cA(I))$ and that $A=\bar\iota\otimes\iota$
  where $\iota:\cB(I)\to\cA(I)$ is the embedding.
  Let $v:=\coev_{\iota}\otimes \id_\iota$ so that $\cA(I)=\iota(\cB(I))v$ by \cite{MR1332979,MR3308880}.
  We get a map $\alpha:\Aut_{\Rep(\cB)}(A)\to \Aut(\cA(I)|\cB(I),\Omega)$ given by $\alpha(g)(\iota(b)v)=\iota(b g^{-1})v$
  and a map $\beta: \Aut(\cA(I)|\cB(I),\Omega)\to\Aut_{\Rep(\cB)}(A)$ given by $\beta(\alpha)=
  (\ev_{\iota}\otimes A)\circ(\id_{\bar\iota}\otimes\alpha^{-1}(v))$.
  One can verify that $\alpha$ and $\beta$ are homomorphisms which are mutually inverse.
  Indeed, by the zig-zag relation we have 
  $(\ev_{\iota}\otimes A)\circ(\id_{\bar\iota}\otimes\alpha(g)(v))=g$, so 
  $\beta\circ\alpha=\id_{\Aut_{\Rep(\cB)}(A)}$.
  Using the other zig-zag relation gives  $\alpha\circ\beta=\id_{\Aut(\cA(I)|\cB(I),\Omega)}$.
\end{proof}
For every $\alpha\in\Aut(\cA)$ with $\alpha(\cB(I))=\cB(I)$, we get an automorphism $\res_{\cB}(\alpha)\in\Aut(\cB)$.
Given a finite global symmetry group $G\subset \Aut(\cB)$, M\"uger \cite{MR2183964} showed that
there is:
\begin{itemize}
  \item 
  an irreducible rational subnet $\cB^G\subset \cB$ where
  $\cB^G(I):=\{x\in\cB(I):\alpha(x)=x \text{ for all } \alpha\in G\}$ \cite{MR1806798} which is rational by \cite{MR1838752},
  \item 
  an action $\underline\rho:\underline{G}\to \uEqBr(\cC)$, $\alpha\mapsto {}^\alpha\pi$ given 
  by ${}^\alpha\pi_I=\alpha\circ\pi_I\circ\alpha^{-1}$ for any representation $\pi=\{\pi_I\}$, and
  \item 
  a $G$-crossed braided extension $\cF=G\mathrm{-}\Rep^I(\cB)$ of $\cC$ compatible with $\underline{\rho}$.
\end{itemize}
Furthermore, $\Rep(\cB^G)$ is braided tensor equivalent to $\cF^G$. 
Thus the ``orbifold subnet'' $\cB^G\subset \cB$ corresponds to gauging $\cC$ by $G$.

The language of tensor categories gives a dictionary between anyons and conformal nets.
We can ask what the notion of an unbroken categorical symmetry means in the setting of conformal nets.
A similar problem has been studied in \cite{MR1228529}.
\begin{defn}
An extension of the global symmetry $G\subset \Aut(\cB)$ of $\cB$ to a global symmetry of $\cA\supset \cB$ corresponds to an embedding $s:G\hookrightarrow \Aut(\cA |\cB^G)\subset\Aut(\cA)$ such that 
$s(g)(\cB(I))=\cB(I)$ for all $g\in G$ and $\res_{\cB}\circ s =\id_G$.
\end{defn}
\begin{prop}
  \label{prop:CNets}
  Let $\cB$ be a rational conformal net,
  $\cC:= \Rep(\cB)$ its UMTC of representations, 
  $G\subset\Aut(\cB)$ a finite subgroup,
  and $A\in\cC$ a connected \'etale $\rm C^*$ Frobenius algebra with associated local extension $\cA\supset\cB$.
  An extension of the global symmetry $G$ of $\cB$ to $\cA$ exists only if 
  the associated action $\underline G \to \uEqBr(\cC)$ lies in $\uEqBrStab(A)$.
  In this case, extensions are in one-to-one correspondence with splittings $\sigma\colon G\to \Aut_{\cC^G}(I(A))$ of the exact sequence \eqref{eq:ShortExactSequence}.
\end{prop}
\begin{proof}
Let us fix an interval $I\in\cI$ and
let us denote 
$\iota_\cA:\cB(I)\hookrightarrow\cA(I)$ the canonical inclusion map.
  Let us assume an extension of $G$ to $\cA$ exists. 
  Then $\iota_\cA\circ\alpha=s(\alpha)\circ\iota_\cA$ for any $\alpha\in G$ which implies that 
  the action of $G$ stabilizes the algebra $A\cong\bar\iota_\cA\otimes\iota_\cA$.
  

Recall that $\Aut(\cB|\cB^G)=\{\alpha_g:g\in G\}$ and that the associated 
categorical action on $\cC=\Rep^I(\cB)$ is given by $g(\tau)={}^{\alpha_g}\tau =\alpha_g\otimes \tau\otimes\alpha_g$ for $\tau\in\Rep^I(\cB)$, where the tensor product correspond to composition of endomorphisms of $\cB(I)$.

To prove the second claim, we show that the following diagram commutes:
  $$
\begin{tikzcd}
1 \arrow[r] & \Aut(\cA|\cB) \arrow[r, hook,] \arrow[d,->,"\beta",shift left] & \Aut(\cA|\cB^G) \arrow[r, "\res_{\cB}"] \arrow[d,<-,"\tilde\alpha"] & \Aut(\cB|\cB^G) \arrow[r]  \arrow[d,equal] & 1 \\
1 \arrow[r]  & \Aut_{\cC}(A) \arrow[r,"\iota"] \arrow[u,"\alpha",shift left]& \Aut_{\cC^G}(I(A)) \arrow[r,"\pi"] & G \arrow[r]  & 1.
\end{tikzcd}
  $$

Recall from Definition \ref{defn:AlgebraI(A)} that $I(A)=\bigoplus_g g(A)$. 
Let $\tilde A$ be the algebra in $\Rep^I(\cB^G)$ corresponding to $\cB^G\subset \cA$.
Then $\tilde A$ is of the form $\bar\iota_{\cB} \otimes \bar\iota_\cA\otimes \iota_{\cA}\otimes \iota_{\cB}$
where $\iota_{\cB}:\cB^G(I)\hookrightarrow \cB(I)$ is the inclusion map. 
Let $f\in \Aut_{\cC}(I(A))$.

We obtain a map $\tilde f:\tilde A \to \tilde A$ as follows. 
As in Lemma \ref{lem:UniqueNonzeroComponent} (2) there is a unique $g$, such that $f_{e,g}\neq 0$ and $f:g(A)\to A$.
Note that $\cB^G(I)\subset \cB(I)$ being a depth two inclusion implies that $\Hom(\alpha_g\otimes \iota_{\cB^G},\alpha_g\otimes \iota_{\cB^G})\cong \bbC$. Therefore
$\bar\iota_{\cB}\otimes f_{e,g}\otimes \iota_{\cB}:
\bar\iota_\cB\otimes \alpha_g\otimes \bar\iota_\cA\otimes\iota_\cA
\otimes \alpha_{g}^{-1}\otimes \iota_{\cB}\to\tilde A$ gives a map $\tilde f:\tilde A\to\tilde A$ defined up to a scalar which can be fixed by demanding that 
$(\bar \iota_\cB \otimes \ev_{\iota_\cA} \otimes \iota_{\cB})\circ \tilde f$
is an automorphism of the algebra associated with $\bar\iota_\cB\otimes \iota_\cB$.
We define $\tilde\alpha(f) = \alpha(\tilde f)$ with $\alpha:\Aut_\cC(A)\to\Aut(\cA|\cB)$ from Lemma \ref{lem:AutomorphismsOfCNets}.
Then $\tilde\alpha(\iota(f))$ equals $\alpha(f)\in\Aut(\cA|\cB^G)$ for all $f\in \Aut_{\cC}(A)$. Thus the first square commutes.

Finally, recall that $\pi(f)=g$ with $g\in G$ the unique element such that $f_{e,g}\neq 0$. 
We claim that $\res_\cB(\tilde\alpha(f))=\alpha_g$. 
Namely, let $b=\iota_\cB(c)v_\cB\in \cB(I)$, where 
$v_{(-)}=\coev_{\iota_{(-)}}\otimes \id_{\iota_{(-)}}$, then 
\begin{align*}
  \alpha(\tilde f)(\iota_\cA(b))&=
  \alpha(\tilde f)(\iota_\cA(\iota_\cB(c)v_\cB)
=
\alpha(\tilde f)(\iota_\cA(\iota_\cB(c)\bar\iota_\cB(\ev_{\iota_\cA})v_\cB)v_\cA)
\\&=
\iota_\cA(\iota_\cB(c)\bar\iota_\cB(\ev_{\iota_\cB})\tilde f^{-1}v_\cB)v_\cA)
=\iota_\cA(\alpha_g(b)).
\end{align*}
Hence the second square commutes, and we are finished.
\end{proof}

\bibliographystyle{amsalpha}
{\footnotesize{
\bibliography{bibliography}
}}
\end{document}